\documentclass[times,sort&compress,3p]{elsarticle}
\journal{}

\usepackage{amsmath,amsfonts,amssymb,amsthm,booktabs,color,
  graphicx,hyperref,url}

\usepackage[labelfont=bf]{caption}

\usepackage{bm,multirow,xcolor,colortbl,rotating,grffile,
  appendix}

\makeatletter
\def\ps@pprintTitle{%
 \let\@oddhead\@empty
 \let\@evenhead\@empty
 \def\@oddfoot{}%
 \let\@evenfoot\@oddfoot}
\makeatother

\newcommand{\eps}{\varepsilon}
\newcommand{\N}{\mathbb{N}}
\newcommand{\R}{\mathbb{R}}
\newcommand{\Z}{\mathbb{Z}}
\newcommand{\dd}{\mathrm{d}}
\newcommand{\Ab}{\mathbb{A}}
\newcommand{\Cb}{\mathbb{C}}
\newcommand{\Db}{\mathbb{D}}
\newcommand{\Gb}{\mathbb{G}}
\newcommand{\Cc}{\mathcal{C}}

\newcommand{\Fc}{\mathcal{F}}
\newcommand{\Ic}{\mathcal{I}}
\newcommand{\Uc}{\mathcal{U}}
\newcommand{\Xc}{\mathcal{X}}

\newcommand{\Ex}{\mathbb{E}}
\newcommand{\Var}{\mathrm{Var}}

\newcommand{\1}{\mathbf{1}}
\newcommand{\ip}[1]{\lfloor #1 \rfloor}
\renewcommand{\Pr}{\mathrm{Pr}}
\newcommand{\BL}{\mathrm{BL}}
\newcommand{\p}{\overset{\Pr}{\to}}

\theoremstyle{plain}
\newtheorem{thm}{Theorem}
\newtheorem{cor}{Corollary}
\newtheorem{lem}{Lemma}
\newtheorem{cond}{Condition}

\theoremstyle{definition}
\newtheorem{remark}{Remark}

\begin{document}

\begin{frontmatter}

\title{Subsampling (weighted smooth) empirical copula processes}

\author[A1]{Ivan Kojadinovic\corref{mycorrespondingauthor}}
\author[A1,A2]{Kristina Stemikovskaya}

\address[A1]{CNRS / Universit\'e de Pau et des Pays de l'Adour / E2S UPPA, Laboratoire de math\'ematiques et applications -- IPRA, UMR 5142, B.P. 1155, 64013 Pau Cedex, France.}
\address[A2]{Universidad del País Vasco, Intelligent Systems Group, Campus de Guipuzcoa, 20018 Donostia - San Sebastian, Spain.}

\cortext[mycorrespondingauthor]{Corresponding author. Email address: \url{ivan.kojadinovic@univ-pau.fr}}

\begin{abstract}
A key tool to carry out inference on the unknown copula when modeling a continuous multivariate distribution is a nonparametric estimator known as the empirical copula. One popular way of approximating its sampling distribution consists of using the multiplier bootstrap. The latter is however characterized by a high implementation cost. Given the rank-based nature of the empirical copula, the classical empirical bootstrap of Efron does not appear to be a natural alternative, as it relies on resamples which contain ties. The aim of this work is to investigate the use of subsampling in the aforementioned framework. The latter consists of basing the inference on statistic values computed from subsamples of the initial data. One of its advantages in the rank-based context under consideration is that the formed subsamples do not contain ties. Another advantage is its asymptotic validity under minimalistic conditions. In this work, we show the asymptotic validity of subsampling for several (weighted, smooth) empirical copula processes both in the case of serially independent observations and time series. In the former case, subsampling is observed to be substantially better than the empirical bootstrap and equivalent, overall, to the multiplier bootstrap in terms of finite-sample performance.
\end{abstract}

\begin{keyword}
  Delete-$h$ jackknife \sep
  Empirical beta copula \sep
  Empirical checkerboard copula \sep
  Rank-based inference \sep
  Weighted weak convergence.
\MSC[2010] Primary 62G09 \sep 
Secondary 62G20 
\end{keyword}

\end{frontmatter}


\section{Introduction}

Let $\bm \Xc_n$ denote a stretch $\bm X_1,\dots,\bm X_n$ from a stationary time series $(\bm X_i)_{i \in \Z}$ of $d$-dimensional random vectors. The distribution function (d.f.) of each $\bm X_i$ is denoted by $F$ and is assumed to have continuous univariate margins $F_1,\dots,F_d$. By Sklar's theorem \cite{Skl59}, it is then well-known that $F$ can be expressed as
\begin{equation}
\label{eq:sklar}
F(\bm x) = C\{F_1(x_1),\dots,F_d(x_d)\}, \qquad \bm x \in \R^d,
\end{equation}
where $C$ is the unique \emph{copula} (a $d$-dimensional d.f.\ with standard uniform margins) associated with $F$.

Eq.~\eqref{eq:sklar} is at the root of the so-called \emph{copula approach} to the modeling of multivariate continuous distributions, which is increasingly applied in numerous fields such as quantitative risk management \cite{McNFreEmb15}, econometrics \cite{Pat12}, or environmental modeling \cite{SalDeMKotRos07}. Indeed, in order to obtain a parametric estimate of $F$, the decomposition in~\eqref{eq:sklar} suggests to model $F_1, \dots, F_d$ by appropriate univariate parametric d.f.s and $C$ by an adequate parametric copula family. The recent infatuation for such a two-step approach in the literature is mostly due to the fact that it has the potential of providing better estimates of the multivariate d.f.\ $F$ than if a direct classical multivariate approach were used; see, for instance, \cite{HofKojMaeYan18} and the references therein for more details.

The modeling of the univariate margins $F_1,\dots,F_d$ of $F$ can be based on classical statistical inference techniques. Inference on the unknown copula $C$ is, however, typically carried out using specific methods exploiting the two-step nature of the underlying modeling. Among the latter methods, rank-based approaches display particularly good properties \citep[see, e.g.,][]{GenFav07, HofKojMaeYan18}. One of their key ingredients is a nonparametric rank-based estimator of $C$ called the \emph{empirical copula} \citep[see, e.g.,][]{Rus76,Deh79}. In the absence of ties in the component samples of the available data $\bm \Xc_n$, it is natural to define the latter simply as the empirical d.f.\ of the multivariate ranks obtained from $\bm \Xc_n$ scaled by $1/n$. Two smooth versions, with better small-sample properties, are the \emph{empirical checkerboard copula} \citep[see, e.g.,][and the references therein]{GenNes07,GenNesRem17} and the \emph{empirical beta copula} \cite{SegSibTsu17}.

Whatever type of empirical copula is used in inference procedures on the unknown copula $C$ in~\eqref{eq:sklar}, it is almost always necessary to rely on resampling techniques to compute corresponding approximate confidence intervals or p-values. A frequently used approach is the so-called \emph{multiplier bootstrap} \citep[see, e.g.,][]{Sca05,RemSca09}. When $\bm \Xc_n$ consists of $n$ independent and identically distributed (i.i.d.) copies of $\bm X$, \citet{BucDet10} empirically found the latter resampling scheme to have better finite-sample properties than approaches consisting of adapting the empirical (multinomial) bootstrap of \citet{Efr79}. Both the i.i.d.\ version of the multiplier bootstrap and its extension to time series investigated in \cite{BucKoj16} are however characterized by a high implementation cost which may deter their use in copula inference procedures. The main aim of this work is to investigate the use of another resampling technique, known as \emph{subsampling} \cite{PolRom94b,PolRomWol99b}, to carry out inference on the unknown copula~$C$ in~\eqref{eq:sklar}.

In the case of i.i.d.\ data, subsampling consists of taking subsamples of size $b < n$ without replacement from the initial data. The statistic of interest is then recomputed for a large number of such subsamples and its sampling distribution is approximated by the empirical distribution of its subsample values. In the time series case, subsamples are restricted to consecutive observations to preserve serial dependence. Note in passing that, in the i.i.d.\ setting, subsampling is connected to the so-called delete-$h$ jackknife \cite{ShaWu89,Wu90}; see also Section 2.3 in \cite{ShaTu96} and, in particular, Remark 2.1 in \cite{PolRom94b}.

A theoretical advantage of subsampling is its asymptotic validity under minimalistic assumptions, weaker than those of the empirical bootstrap for instance; see \cite{PolRom94b,PolRomWol99,PolRomWol99b} for details. From a practical perspective, subsampling is very simple to implement, its only drawback being the necessity of choosing the subsample size $b$. In the context of copula modeling based on ranks (and, more generally, in the context of rank-based statistics), it is particularly attractive because subsamples do not contain ties unlike, for instance, resamples of size~$n$ in the case of the empirical bootstrap. The latter cannot therefore be directly used for rank-based statistics as shall be discussed, for instance, in Section~\ref{sec:MC} for (certain functionals of) the empirical copula. For this reason, subsampling appears as a simple way to obtain approximations of the sampling distributions of the empirical checkerboard and empirical beta copulas, even in a time series context.

Notice that, in the case of i.i.d.\ observations, subsamples of size $b$ could also be obtained by sampling \emph{with replacement} from $\bm \Xc_n$ \cite[see, e.g.,][]{Swa86,BicGotVan97}. The resulting resampling technique, sometimes referred to as the \emph{$b$ out of $n$ bootstrap} (and which coincides with the empirical bootstrap when $b=n$), thus forms subsamples with ties and therefore suffers from the same inconvenience as the empirical bootstrap in the rank-based context under consideration. For this reason, we shall not investigate it theoretically in this work. We shall however mention this alternative technique again when summarizing the results of our Monte Carlo experiments.

The rest of this article is organized as follows. In the second section, we introduce the main versions of the empirical copula appearing in the literature and define the corresponding empirical copula processes. The third section establishes the asymptotic validity of the subsampling methodology for the latter processes, while the fourth section states weighted versions of such results, thereby providing a first simple way to carry out inference on quantities which can be expressed as functionals of weighted empirical copula processes. The fifth section summarizes the results of Monte Carlo experiments in the i.i.d.\ and time series cases, and provides recommendations for the choice of the subsample size in the i.i.d.\ case. Finally, concluding remarks are gathered in Section~\ref{sec:conc}.

All proofs are deferred to a sequence of appendices. Additional simulation results are provided in a supplement. The following notation is used in the sequel. The arrow~`$\leadsto$' denotes weak convergence in the sense of Definition~1.3.3 in \cite{vanWel96}, and, given a set $T$, $\ell^\infty(T)$ (resp.\ $\Cc(T)$) represents the space of all bounded (resp.\ continuous) real-valued functions on $T$ equipped with the uniform metric. All convergences are for $n\to\infty$ if not mentioned otherwise.

\section{Empirical copulas and empirical copula processes}

It is well-known that the unique copula in~\eqref{eq:sklar} can be expressed \citep[see, e.g.,][]{Skl59,Rus09} as
\begin{equation}
\label{eq:C}
C(\bm u) = F \{F_1^{-} (u_1),\dots,F_d^{-}(u_d) \}, \qquad \bm u \in [0,1]^d,
\end{equation}
where, for any univariate d.f.\ $H$, $H^{-}$ denotes its associated quantile function (generalized inverse) defined by
$$
H^{-}(y) = \inf\{x \in \R : H(x) \geq y \}, \qquad y \in [0,1],
$$
with the convention that $\inf \emptyset = \infty$.

A first natural definition of the empirical copula, due to \citet{Deh79,Deh81}, stems from~\eqref{eq:C} and the plug-in principle. Let $F_n$ denote the empirical d.f.\ of $\bm \Xc_n$ and let $F_{n1},\dots,F_{nd}$ be the corresponding univariate margins. The empirical copula of $\bm \Xc_n$ is then defined by
\begin{equation}
\label{eq:Cn}
C_n(\bm u) = F_n\{F_{n1}^{-}(u_1),\dots,F_{nd}^{-}(u_d) \}, \qquad \bm u \in [0,1]^d.
\end{equation}

Another definition of the empirical copula frequently found in the literature \citep[see, e.g.,][]{GenGhoRiv95} is
\begin{equation}
\label{eq:tilde:Cn}
\tilde C_n(\bm u) = \frac{1}{n} \sum_{i=1}^n \prod_{j=1}^d \1\{F_{nj}(X_{ij}) \leq u_j\}, \qquad \bm u \in [0,1]^d.
\end{equation}

When there are no ties in the components samples of $\bm \Xc_n$, it is well-known that, for any $i \in \{1,\dots,n\}$ and any $j \in \{1,\dots,d\}$,  $n F_{nj}(X_{ij})$ is nothing else then the rank $R_{ij,n}$ of $X_{ij}$ among $X_{1j},\dots,X_{nj}$. In that case, $\tilde C_n$ coincides with the version of the empirical copula appearing in \citet{Rus76} and given by
\begin{equation}
\label{eq:hat:Cn}
\hat C_n(\bm u) = \frac{1}{n} \sum_{i=1}^n \prod_{j=1}^d \1(R_{ij,n}/n \leq u_j), \qquad \bm u \in [0,1]^d.
\end{equation}
The latter is merely the empirical d.f.\ of the sample $(R_{i1,n},\dots,R_{id,n}) / n$, $i \in \{1,\dots,n\}$, of normalized multivariate ranks.

Before proceeding further, let us formally introduce the no-ties condition.

\begin{cond}[No ties]
\label{cond:ties}
For any $j \in \{1,\dots,d\}$, the $j$th component sample $X_{1j},\dots,X_{nj}$ of $\bm \Xc_n$ does not contain ties.
\end{cond}

\begin{remark}
When the available data $\bm \Xc_n$ consist of $n$ i.i.d.\ copies of the random vector of interest $\bm X$, continuity of the marginal d.f.s $F_1,\dots,F_d$ implies that Condition~\ref{cond:ties} is satisfied. In a time series context, however, ties may occur with positive probability even if $F_1,\dots,F_d$ are continuous: as suggested in \cite{BerSeg18}, take, for instance, a Markov chain where the current state is repeated with positive probability.
\end{remark}

Under Condition~\ref{cond:ties}, classical calculations \citep[see, e.g.,][proof of Lemma~4.7]{BerBucVol17}, imply that, almost surely,
\begin{equation}
\label{eq:ae:Cn}
\sup_{\bm u \in [0,1]^d} | C_n(\bm u) - \hat C_n(\bm u) | \leq \frac{d}{n}.
\end{equation}
The relative simplicity, computation-wise, of $\hat C_n$ in~\eqref{eq:hat:Cn} over $C_n$ in~\eqref{eq:Cn} makes it the natural definition in the absence of ties. In the presence of ties, however, $\hat C_n$ is not unambiguously defined but $\tilde C_n$ in~\eqref{eq:tilde:Cn} could still be used as an alternative to $C_n$ in~\eqref{eq:Cn}. Interestingly enough, $C_n$ and $\tilde C_n$ can be shown to remain sufficiently close under a rather minimalistic condition that shall be stated towards the end of this section.

In the absence of ties, the empirical copula, whether it is defined by~\eqref{eq:Cn} or~\eqref{eq:hat:Cn}, is not, however, a genuine copula: it is for instance easy to verify that the univariate margins of $C_n$ and $\hat C_n$ are not standard uniform but only asymptotically standard uniform. In the absence of ties, two smoother alternatives to the empirical copula that are genuine copulas are the \emph{empirical checkerboard copula} and the \emph{empirical beta copula}. The empirical checkerboard copula is merely a multilinear extension of $\hat C_n$ and is defined by
\begin{equation}
\label{eq:Cn:hash}
  C_n^\#(\bm{u}) = \frac{1}{n} \sum_{i=1}^n \prod_{j=1}^d \min\{ \max\{n u_j - R_{ij,n} + 1, 0 \}, 1 \}, \quad \bm u \in [0,1]^d,
\end{equation}
see, e.g., \cite{CarTay02,SegSibTsu17}, and the references therein. It is important to note that the empirical checkerboard copula can also be defined in the presence of ties and even for discontinuous margins $F_1,\dots,F_d$; see, for instance, \cite{GenNes07,GenNesRem14,GenNesRem17}. Coming back to our context of continuous margins, it is easy to verify (see Lemma~\ref{lem:ae:Cn:hash} in Appendix~\ref{app:proofs}) that, under Condition~\ref{cond:ties}, almost surely,
\begin{equation}
\label{eq:ae:Cn:hash}
\sup_{\bm u \in [0,1]^d} | C_n^\#(\bm u) - \hat C_n(\bm u) | \leq \frac{d}{n},
\end{equation}
thereby indicating that the empirical checkerboard copula can be thought of as a smoothing of the empirical copula $\hat C_n$ at bandwidth $O(1/n)$.

The empirical beta copula, proposed by \citet{SegSibTsu17}, is obtained by replacing indicator functions in~\eqref{eq:hat:Cn} by d.f.s of particular beta distributions. Specifically, the empirical beta copula is defined by
\begin{equation}
\label{eq:Cn:beta}
  C_n^\beta(\bm{u}) = \frac{1}{n} \sum_{i=1}^n \prod_{j=1}^d \bm F_{n,R_{ij,n}}(u_j), \quad \bm{u} \in [0,1]^d,\index{$C_n^\beta$}
\end{equation}
where, for any $r \in \{1,\dots,n\}$, $\bm F_{n,r}$ denotes the d.f.\ of the beta distribution with parameters $r$ and $n+1-r$. The empirical beta copula is actually a particular case of the \emph{empirical Bernstein copula} introduced in \cite{SanSat04} and further studied in \cite{JanSwaVer12}, when the degrees of all Bernstein polynomials are set equal to the sample size. Proposition~2.8 in \cite{SegSibTsu17} shows that, under Condition~\ref{cond:ties}, the uniform distance between the empirical beta copula and $\hat C_n$ is $O(n^{-1/2} (\ln n)^{1/2})$, thereby suggesting to see the empirical beta copula as a smoothing of the empirical copula~$\hat C_n$ at approximately bandwidth $O(n^{-1/2})$; see also Corollary 3.7 in \cite{SegSibTsu17}.

The previous definitions give rise to up to five different \emph{empirical copula processes}. The two most studied ones are
\begin{equation}
\label{eq:Cbn}
\Cb_n(\bm u) = \sqrt{n} \{C_n(\bm u) - C(\bm u)\}, \qquad \bm u \in [0,1]^d,
\end{equation}
and its asymptotically equivalent version in the absence of ties given by
\begin{equation}
\label{eq:hat:Cbn}
\hat \Cb_n(\bm u) = \sqrt{n} \{\hat C_n(\bm u) - C(\bm u)\}, \qquad \bm u \in [0,1]^d.
\end{equation}
In the case of i.i.d.\ observations, their weak convergence was investigated for instance in \cite{GanStu87,FerRadWeg04,Tsu05,vanWel07,Seg12}. As we shall see below, the time series case can be elegantly handled using the work of \citet{BucVol13}.

For the sake of completeness, we also define the process
\begin{equation}
\label{eq:tilde:Cbn}
\tilde \Cb_n(\bm u) = \sqrt{n} \{\tilde C_n(\bm u) - C(\bm u)\}, \qquad \bm u \in [0,1]^d,
\end{equation}
whose study becomes of interest only in a time series context in which ties may occur.

In the absence of ties, two smoother empirical copula processes are obtained from the empirical checkerboard copula and the empirical beta copula, namely
\begin{equation}
\label{eq:Cbn:hash}
\Cb_n^\#(\bm u) = \sqrt{n} \{C_n^\#(\bm u) - C(\bm u)\}, \qquad \bm u \in [0,1]^d,
\end{equation}
and
\begin{equation}
\label{eq:Cbn:beta}
\Cb_n^\beta(\bm u) = \sqrt{n} \{C_n^\beta(\bm u) - C(\bm u)\}, \qquad \bm u \in [0,1]^d.
\end{equation}
The former was studied by \citet{GenNesRem17} in a broader context than the one considered in this work, while the latter was investigated by \citet{SegSibTsu17}.

Under the assumption of continuity of the marginals d.f.s $F_1,\dots,F_d$ made in this work, the weak convergence of the aforementioned empirical copula processes can be elegantly stated by invoking the two conditions considered in \cite{BucVol13}. The first condition concerns the weak convergence of the multivariate empirical process based on the (unobservable) sample $\bm U_1,\dots,\bm U_n$ obtained from $\bm \Xc_n$ by (marginal) probability integral transformations, where
\begin{equation}
\label{eq:Ui}
\bm U_i = (F_1(X_{i1}),\dots,F_d(X_{id})), \qquad i \in \Z.
\end{equation}
Notice that $\bm U_1,\dots,\bm U_n$ is a sample from $C$ and let $G_n$ denote its empirical d.f. The multivariate empirical process based on $\bm U_1,\dots,\bm U_n$ is then
\begin{equation}
\label{eq:Gbn}
\Gb_n (\bm u)= \sqrt{n} \{ G_n(\bm u) - C(\bm u) \}, \qquad \bm u \in [0,1]^d.
\end{equation}

\begin{cond}[Weak convergence of $\Gb_n$]
\label{cond:Gbn}
The multivariate empirical process $\Gb_n$ in~\eqref{eq:Gbn} converges weakly in $\ell^\infty([0,1]^d)$ to a tight, centered Gaussian process $\Gb_C$ concentrated on
\begin{equation}
\label{eq:Cc0}
\Cc_0 = \{ f \in \Cc([0,1]^d) : f(1,\dots,1) = 0 \text{ and } f(\bm u) = 0 \text{ if some components of $\bm u$ are equal to } 0 \}.
\end{equation}
\end{cond}


The second condition was initially introduced by \citet{Seg12} and is nonrestrictive in the sense that it is necessary for the candidate weak limit of the empirical copula processes under consideration to exist pointwise and have continuous sample paths.

\begin{cond}[Continuous partial derivatives]
\label{cond:partial:C}
For any $j \in \{1,\dots,d\}$, the $j$th partial derivative $\dot C_j(\bm u) = \partial C(\bm u)/\partial u_j$ of $C$ exists and is continuous on the set
\begin{equation}
\label{eq:Vj}
V_j = \{\bm u \in [0,1]^d : u_j \in (0,1) \}.
\end{equation}
\end{cond}

From Corollary 2.5 in \cite{BucVol13}, we then know that, under Conditions~\ref{cond:Gbn} and~\ref{cond:partial:C}, the empirical copula process $\Cb_n$ in~\eqref{eq:Cbn} converges weakly in $\ell^\infty([0,1]^d)$ to a tight, centered Gaussian process $\Cb_C$ which may be expressed in terms of $\Gb_C$ as
\begin{equation}
\label{eq:CbC}
\Cb_C(\bm u) = \Gb_C(\bm u) - \sum_{j=1}^d \dot C_j(\bm u) \Gb_C(\bm u^{(j)}), \qquad \bm u \in [0,1]^d,
\end{equation}
where, for any $j \in \{1,\dots,d\}$ and $\bm u \in [0,1]^d$, $\bm u^{(j)}$ is the vector of $[0,1]^d$ whose components are all equal to 1 except the $j$th which is equal to $u_j$, and with the convention that $\dot C_j(\bm u)$ is equal to 0 if $u_j \in \{0,1\}$.

Under Condition~\ref{cond:Gbn}, proceeding as in the proof of Lemma~4.7 in \cite{BerBucVol17} and using the asymptotic uniform equicontinuity in probability of $\Gb_n$, we immediately obtain that
\begin{equation}
\label{eq:ae:Cn:tilde}
\sup_{\bm u \in [0,1]^d} | C_n(\bm u) - \tilde C_n(\bm u) | = o_\Pr(n^{-1/2}),
\end{equation}
which, if Condition~\ref{cond:partial:C} also holds, is sufficient to conclude that $\tilde \Cb_n$ in~\eqref{eq:tilde:Cbn} converges weakly in $\ell^\infty([0,1]^d)$ to $\Cb_C$ in~\eqref{eq:CbC} as well. Assuming additionally Condition~\ref{cond:ties}, that is, the absence of ties, the same holds for $\hat \Cb_n$ (since, in that case, $\hat \Cb_n = \tilde \Cb_n$), as well as for $\Cb_n^\#$ in~\eqref{eq:Cbn:hash} and $\Cb_n^\beta$ in~\eqref{eq:Cbn:beta} by~\eqref{eq:ae:Cn:hash} and Corollary~3.7 in \cite{SegSibTsu17}, respectively. In other words, the empirical copula processes under consideration have the same weak limit under Conditions~\ref{cond:ties},~\ref{cond:Gbn} and~\ref{cond:partial:C}.

\section{Subsampling empirical copula processes}

We start by describing the considered subsampling framework before stating a theorem establishing the asymptotic validity of the subsampling methodology for the empirical copula processes introduced in the previous section.

Because we shall in part rely on the very general results of \citet{PolRomWol99}, we assume that the available sample $\bm \Xc_n$ is a stretch from a \emph{strongly mixing} stationary sequence $(\bm X_i)_{i \in \Z}$. Denote by $\Fc_j^k$ the $\sigma$-field generated by $(\bm X_i)_{j \leq i \leq k}$, $j, k \in \Z \cup \{-\infty,+\infty \}$, and recall that the strong mixing coefficients corresponding to the stationary sequence $(\bm X_i)_{i \in \Z}$ are then defined by $\alpha_0^{\bm X} = 1/2$, 
\begin{equation*}
\alpha_r^{\bm X} = \sup_{A \in \Fc_{-\infty}^0,B\in \Fc_{r}^{+\infty}} \big| \Pr(A \cap B) - \Pr(A) \Pr(B) \big|, \qquad r \in \N, \, r > 0,
\end{equation*}
and that the sequence $(\bm X_i)_{i \in \Z}$ is said to be \emph{strongly mixing} if $\alpha_r^{\bm X} \to 0$ as $r \to \infty$.

We consider two settings for the sequence $(\bm X_i)_{i \in \Z}$:
\begin{description}
\item[i.i.d.] The coefficients $\alpha_r^{\bm X}, r \geq 1$, are all equal to zero implying that the stretch $\bm \Xc_n$ consists of i.i.d.\ random vectors.
\item[s.m.] The sequence $(\bm X_i)_{i \in \Z}$ is strongly mixing but not i.i.d.\ with $\alpha^{\bm X}_r = O(r^{-a})$ as $r \to \infty$, for some $a > 1$.
\end{description}

\begin{remark}
The condition on the mixing coefficients stated in the s.m.\ setting is among the weakest possible ones and, from Theorem~1 of \citet{Buc15}, implies that Condition~\ref{cond:Gbn} is then satisfied. Note that Condition~\ref{cond:Gbn} also obviously holds under the i.i.d.\ setting as a consequence of Donsker's theorem.
\end{remark}

Let $b < n$ denote the size of subsamples which will be obtained from $\bm \Xc_n$. Under the i.i.d.\ setting, subsamples can be formed simply by sampling without replacement from~$\bm \Xc_n$. The number of possible subsamples is thus $N_{b,n} = \binom{n}{b}$. Following \citet{PolRom94b}, the subsampling methodology, say for $\Cb_n$ in~\eqref{eq:Cbn}, consists of, first, evaluating a computable version of $\Cb_n$ for a large number of the $N_{b,n}$ subsamples $\bm \Xc_b^{[m]}$, $m \in \{1,\dots,N_{b,n}\}$, of size $b$ obtained by sampling without replacement from~$\bm \Xc_n$ (considering all of the $N_{b,n}$ subsamples is typically infeasible in practice), and, second, of carrying out the inference on~$C$ using these \emph{subsample replicates} of $\Cb_n$. The latter are naturally formed as follows. For $m \in \{1,\dots,N_{b,n}\}$, let $F_b^{[m]}$ be the empirical d.f.\ of the sample $\bm \Xc_b^{[m]}$, let $F_{b1}^{[m]},\dots,F_{bd}^{[m]}$ be the univariate margins of $F_b^{[m]}$, and let
$$
C_b^{[m]}(\bm u) = F_b^{[m]}\{F_{b1}^{[m],-}(u_1),\dots,F_{bd}^{[m],-}(u_d) \}, \qquad \bm u \in [0,1]^d,
$$
be the empirical copula of $\bm \Xc_b^{[m]}$. The subsample replicates of the empirical copula process $\Cb_n$ are then defined by
\begin{equation}
\label{eq:Cbbm}
\Cb_b^{[m]}(\bm u) = \sqrt{b} \{C_b^{[m]}(\bm u) - C_n(\bm u)\}, \qquad \bm u \in [0,1]^d, \, m \in \{1,\dots,N_{b,n}\}.
\end{equation}
Hence, a replicate of $\Cb_n$ for a subsample $\bm \Xc_b^{[m]}$ coincides with $\Cb_b$ on $\bm \Xc_b^{[m]}$ up to the centering term which is $C_n$ instead of the unknown copula $C$.

\begin{remark}
  \label{rem:finite:pop:cor}
Both for theoretical and practical reasons, it is often meaningful to correct subsample replicates by multiplying them by the \emph{finite population correction} $(1- b/n)^{-1/2}$~; see \citet[Section~10.3.1]{PolRomWol99b} and \citet[Section~2.3.1]{ShaTu96}. This factor arises from the analysis of subsampling for the mean: in the case of univariate i.i.d.\ observations $X_1,\dots,X_n$, the variance of uncorrected subsample replicates can be verified to be $(1-b/n)$ times the variance of $\sqrt{n} \{\bar X_n - \Ex(X)\}$, where $\bar X_n = n^{-1} \sum_{i=1}^n X_i$. The discussion in \citet[Chapter~10]{PolRomWol99b} in the case of the mean suggests that the use of the finite population correction may always be beneficial in finite samples. 
\end{remark}

Coming back to our setting, given a subsample replicate $\Cb_b^{[m]}$ of $\Cb_n$, we define its {\em corrected} version to be
\begin{equation}
\label{eq:Cbbm:corrected}
\Cb_{b,c}^{[m]} = (1-b/n)^{-1/2} \, \Cb_b^{[m]}.
\end{equation}
To fix ideas further, assume that we are interested in estimating a linear functional $\psi(C)$ of the unknown copula $C$ (such as Spearman's rho for instance). An approximate confidence interval of expected asymptotic level $1-\alpha$ for $\psi(C)$ based on subsampling $\Cb_n$ is then given by
$$
\Ic_{n,N_{b,n},\alpha} = \Big[ \psi(C_n) - n^{-1/2} \Fc_{N_{b,n}}^{-}(1-\alpha/2), \psi(C_n) - n^{-1/2} \Fc_{N_{b,n}}^{-}(\alpha/2) \Big], \qquad \alpha \in (0,1/2),
$$
where $\Fc_{N_{b,n}}$ is the empirical d.f.\ of the sample of the $N_{b,n}$ (corrected) subsample replicates $\psi(\Cb_{b,c}^{[m]})$, $m \in \{1,\dots,N_{b,n}\}$, of $\psi(\Cb_n)$. The interval $\Ic_{n,N_{b,n},\alpha}$ is nothing else than the subsampling version of the so-called \emph{basic bootstrap confidence interval} \citep[see, e.g.,][Chapter~5]{DavHin97}. Since, as already mentioned, $N_{b,n}$ is typically too large for $\Fc_{N_{b,n}}$ to be evaluated, one generally needs to rely on a stochastic approximation. The latter typically consists of choosing independently $M$ integers $I_{1,n},\dots,I_{M,n}$ with replacement from $\{1,\dots,N_{b,n}\}$ and proceeding as previously using the subsample replicates $\psi(\Cb_{b,c}^{[I_{m,n}]})$, $m \in \{1,\dots,M\}$. 

Under the s.m.\ setting, the only difference is that the approach has to be restricted to subsamples of size $b$ consisting of consecutive observations so that the serial dependence appearing in $\bm \Xc_n$ is partly preserved. Hence, in that case, the number of possible subsamples is $N_{b,n} = n - b +1 $ and a computable version of the empirical copula process of interest has to be evaluated on subsamples $\bm \Xc_b^{[m]}$ of the form $\bm X_m,\dots,\bm X_{m+b-1}$, $m \in \{1,\dots,N_{b,n}\}$.

The asymptotic validity of the subsampling methodology for the empirical copula processes introduced in the previous section is established under conditions which are very close to the weakest possible ones necessary for the weak convergence of these processes. These rather minimalistic conditions are not surprising as explained in Remark~\ref{rem:simple} below.  For any $m \in \{1,\dots,N_{b,n}\}$, let $\tilde C_b^{[m]}$, $\hat C_b^{[m]}$, $C_b^{\#,[m]}$ and $\hat C_b^{\beta,[m]}$ be the versions of $\tilde C_n$ in~\eqref{eq:tilde:Cn}, $\hat C_n$ in~\eqref{eq:hat:Cn}, $C_n^\#$ in~\eqref{eq:Cn:hash} and $C_n^\beta$ in~\eqref{eq:Cn:beta}, respectively, computed from the subsample $\bm \Xc_b^{[m]}$. The subsamples replicates of $\tilde \Cb_n$, $\hat \Cb_n$, $\Cb_n^\#$ and $\Cb_n^\beta$ are then respectively defined by
\begin{align}
\label{eq:tilde:Cbbm}
\tilde \Cb_b^{[m]}(\bm u) &= \sqrt{b} \{\tilde C_b^{[m]}(\bm u) - \tilde C_n(\bm u)\}, \\
\label{eq:hat:Cbbm}
\hat \Cb_b^{[m]}(\bm u) &= \sqrt{b} \{\hat C_b^{[m]}(\bm u) - \hat C_n(\bm u)\}, \\
\label{eq:Cbbm:hash}
\Cb_b^{\#,[m]}(\bm u) &= \sqrt{b} \{C_b^{\#,[m]}(\bm u) - C_n^\#(\bm u)\}, \\
\label{eq:Cbbm:beta}
\Cb_b^{\beta,[m]}(\bm u) &= \sqrt{b} \{C_b^{\beta,[m]}(\bm u) - C_n^\beta(\bm u)\},
\end{align}
for $\bm u \in [0,1]^d$ and $m \in \{1,\dots,N_{b,n}\}$, and their corrected versions $\tilde \Cb_{b,c}^{[m]}$, $\hat \Cb_{b,c}^{[m]}$, $\Cb_{b,c}^{\#,[m]}$ and $\Cb_{b,c}^{\beta,[m]}$, respectively, are defined analogously to~\eqref{eq:Cbbm:corrected}. The following theorem is proven in Appendix~\ref{app:proofs}.

\begin{thm}[Subsampling empirical copula processes]
\label{thm:sub:Cbn}
Assume that Condition~\ref{cond:partial:C} holds and that $b = b_n \to \infty$. Also, let $I_{1,n}$ and $I_{2,n}$ be independent random variables, independent of $\Xc_n$ and uniformly distributed on the set $\{1,\dots,N_{b,n}\}$.
\begin{enumerate}[(i)]
\item Under the i.i.d.\ setting, if $b/n \to \alpha \in [0,1)$, then
\begin{align}
\label{eq:sub:Cbn:fsc}
(\Cb_n, \Cb_{b,c}^{[I_{1,n}]}, \Cb_{b,c}^{[I_{2,n}]}) &\leadsto (\Cb_C,\Cb_C^{[1]},\Cb_C^{[2]}), \\
\label{eq:sub:hat:Cbn:fsc}
  (\hat \Cb_n, \hat \Cb_{b,c}^{[I_{1,n}]}, \hat \Cb_{b,c}^{[I_{2,n}]}) &\leadsto (\Cb_C,\Cb_C^{[1]},\Cb_C^{[2]}), \\
\label{eq:sub:Cbn:hash:fsc}
(\Cb_n^\#, \Cb_{b,c}^{\#,[I_{1,n}]}, \Cb_{b,c}^{\#,[I_{2,n}]}) &\leadsto (\Cb_C,\Cb_C^{[1]},\Cb_C^{[2]}), \\
\label{eq:sub:Cbn:beta:fsc}
(\Cb_n^\beta, \Cb_{b,c}^{\beta,[I_{1,n}]}, \Cb_{b,c}^{\beta,[I_{2,n}]}) &\leadsto (\Cb_C,\Cb_C^{[1]},\Cb_C^{[2]}),
\end{align}
in $\{ \ell^\infty([0,1]^d) \}^3$, where $\Cb_C^{[1]}$ and $\Cb_C^{[2]}$ are independent copies of $\Cb_C$ in~\eqref{eq:CbC}.

\item Under the s.m.\ setting, if $b/n \to 0$,
\begin{align}
\label{eq:sub:Cbn}
(\Cb_n,\Cb_b^{[I_{1,n}]},\Cb_b^{[I_{2,n}]}) &\leadsto (\Cb_C,\Cb_C^{[1]},\Cb_C^{[2]}), \\
\label{eq:sub:tilde:Cbn}
(\tilde \Cb_n, \tilde \Cb_b^{[I_{1,n}]}, \tilde \Cb_b^{[I_{2,n}]}) &\leadsto (\Cb_C,\Cb_C^{[1]},\Cb_C^{[2]}),
\end{align}
in $\{ \ell^\infty([0,1]^d) \}^3$. 
If Condition~\ref{cond:ties} additionally holds,
\begin{align}
\label{eq:sub:hat:Cbn}
(\hat \Cb_n, \hat \Cb_b^{[I_{1,n}]}, \hat \Cb_b^{[I_{2,n}]}) &\leadsto (\Cb_C,\Cb_C^{[1]},\Cb_C^{[2]}), \\
\label{eq:sub:Cbn:hash}
(\Cb_n^\#,\Cb_b^{\#,[I_{1,n}]},\Cb_b^{\#,[I_{2,n}]}) &\leadsto (\Cb_C,\Cb_C^{[1]},\Cb_C^{[2]}), \\
\label{eq:sub:Cbn:beta}
(\Cb_n^\beta,\Cb_b^{\beta,[I_{1,n}]},\Cb_b^{\beta,[I_{2,n}]}) &\leadsto (\Cb_C,\Cb_C^{[1]},\Cb_C^{[2]}),
\end{align}
in $\{ \ell^\infty([0,1]^d) \}^3$.
\end{enumerate}
\end{thm}

\begin{remark}
Choosing $b$ such that $b/n \to 0$ as in (ii) is the usual assumption found in asymptotic validity results for the subsampling methodology; see \cite{PolRom94b,PolRomWol99,PolRomWol99b}. Imposing that $b/n \to \alpha \in (0,1)$ seems only possible in the i.i.d.\ setting in ``mean-like'' situations such as those considered in \cite{Wu90,PraWel93}; see also Remark~2.2.3 in \cite{PolRomWol99b}. In the case of serially dependent observations, \citet{Lah01} showed that, in the case of the sample mean, the condition $b/n \to 0$ is necessary; see also Remark~3.2.2 in \cite{PolRomWol99b}.
\end{remark}

\begin{remark}
  \label{rem:simple}
  Apart from the assumptions related to the subsampling methodology, no additional conditions than those necessary for the weak convergence of the empirical copula processes are involved in the theorem. For $b/n \to 0$, this is not surprising in view of the general results obtained in \cite{PolRom94b,PolRomWol99} which state that the subsampling methodology is asymptotically valid under minimalistic conditions, weaker than those necessary for the asymptotic validity of the empirical bootstrap, for example. As far as the empirical copula process $\hat \Cb_n$ is concerned for instance, the assumptions of Theorem~\ref{thm:sub:Cbn} are indeed weaker than, for example, those of Proposition~4.2 in \cite{BucKoj16} on the dependent multiplier bootstrap for $\hat \Cb_n$.
\end{remark}

\begin{remark}
The proof of (i) relies in part on the key results of \citet{PraWel93} on the exchangeable bootstrap for the general empirical process; see also \citet[Section~3.6]{vanWel96}. The latter are specialized therein to obtain the asymptotic validity of the delete-$h$ jackknife \cite{ShaWu89,Wu90} for the general empirical process. The proof of (ii) is essentially a consequence of the general result on subsampling stated in Theorem 3.1 in \cite{PolRomWol99}.
\end{remark}

\begin{remark}
  All the weak convergences stated in~(i) involve corrected subsample replicates obtained, as explained in Remark~\ref{rem:finite:pop:cor}, by multiplying the initial replicates by the finite population correction $(1 - b/n)^{-1/2}$~; see, for instance,~\eqref{eq:Cbbm:corrected}. The finite population correction is not needed in the weak convergences stated in (ii) because $(1 - b/n)^{-1/2}$ tends to~1 under the assumption that $b/n \to 0$. It cannot, however, be dispensed with under the i.i.d.\ setting if $b / n \to \alpha \in (0,1)$; see also \citet[Example~3.6]{PraWel93}. From a practical perspective, as already mentioned in Remark~\ref{rem:finite:pop:cor}, the discussion in \citet[Chapter~10]{PolRomWol99b} (in the case of the mean) suggests that the use of the finite population correction may always be beneficial in finite samples. Our Monte Carlo experiments for empirical copula processes, whose results are partly reported in Section~\ref{sec:MC}, essentially corroborate that claim.
\end{remark}

\begin{remark}
The convergence results stated in (i) and (ii) establish the asymptotic validity of the subsampling methodology for the empirical copula processes of interest by stating their weak convergence jointly with two subsample replicates. By Lemma~3.1 in \cite{BucKoj18}, these unconditional asymptotic validity results are equivalent to more classical conditional results which rely, however, on a more subtle mode of convergence. For instance,~\eqref{eq:sub:Cbn:fsc} can be equivalently informally stated as ``$\Cb_{b,c}^{[I_{1,n}]}$ converges weakly to $\Cb_C$ in $\ell^\infty([0,1]^d)$ conditionally on $\bm \Xc_n$ in probability''; see, e.g., \cite[Section~2.2.3]{Kos08} or \cite{BucKoj18} for a precise definition of that mode of convergence in terms of an appropriate version of the bounded Lipschitz metric.
\end{remark}

\section{Subsampling weighted empirical copula processes}

Empirical copula processes were recently shown to converge weakly also with respect to stronger metrics than the supremum distance. A first seminal result in that direction is due to \citet{BerBucVol17} for the empirical copula processes $\Cb_n$ and $\tilde \Cb_n$. \citet{BerSeg18} have shown a similar result for the empirical beta copula process $\Cb_n^\beta$. Because the latter involves the empirical beta copula $C_n^\beta$, which is a genuine copula, its statement is simpler and takes the form of a weighted weak convergence in $\ell^\infty([0,1]^d)$ (that is, with respect to the uniform metric). The weight function considered in the aforementioned references is
\begin{equation}
\label{eq:g}
g(\bm u) = \bigwedge_{j=1}^d \left\{u_j \wedge \bigvee_{k = 1 \atop k \neq j}^d (1-u_k) \right\}, \qquad \bm u \in [0,1]^d,
\end{equation}
where $\wedge$ and $\vee$ denote the minimum and maximum operators, respectively. The corresponding weighted weak convergence results were proven under the two following additional conditions.

\begin{cond}[Exponential mixing]
\label{cond:mixing}
The sequence $(\bm X_i)_{i \in \Z}$ is strongly mixing with $\alpha_r^{\bm X} = O(a^r)$ as $r \to \infty$, for some $a \in (0,1)$.
\end{cond}

\begin{cond}[Smooth second-order partial derivatives]
\label{cond:partial2:C}
For any $j_1,j_2 \in \{1,\dots,d\}$, the second-order partial derivative $\ddot C_{j_1j_2}(\bm u) = \partial^2 C(\bm u)/\partial u_{j_1}\partial u_{j_2}$ of $C$ exists and is continuous on the set $V_{j_1} \cap V_{j_2}$, where $V_j$ is defined by~\eqref{eq:Vj}. Moreover, there exists a constant $\kappa > 0$ such that
$$
|\ddot C_{j_1j_2}(\bm u)| \leq \kappa \min \left\{\frac{1}{u_{j_1}(1-u_{j_1})}, \frac{1}{u_{j_2}(1-u_{j_2})}\right\}, \qquad \forall \, \bm u \in V_{j_1} \cap V_{j_2}.
$$
\end{cond}

Note that Condition~\ref{cond:partial2:C} first appeared in \cite{Seg12} where it was used to prove the almost sure representation for $\Cb_n$ originally conjectured in \cite{Stu84}. As discussed in \cite{BerSeg18}, this condition is satisfied for several commonly occurring copulas.

Theorem~2 in \cite{BerSeg18} then states that, under Conditions~\ref{cond:ties},~\ref{cond:partial:C}, \ref{cond:mixing} and~\ref{cond:partial2:C}, for any $\omega \in [0,1/2)$, the weighted empirical beta copula process $\Cb_n^\beta/g^\omega$ converges weakly in $\ell^\infty([0,1]^d)$ to  $\Cb_C/g^\omega$, where $\Cb_C$ and $g$ are given in~\eqref{eq:CbC} and~\eqref{eq:g}, respectively. In the previous statement, since the zero-set of $\Cb_n^\beta$ includes the zero-set of~$g$, $\Cb_n^\beta/g^\omega$ is taken to be zero as soon as $g = 0$ by convention. The previous result relies in part on Theorem~2.2 in \cite{BerBucVol17} which provides a similar weighted weak convergence for $\Cb_n$ in~\eqref{eq:Cbn} and $\tilde \Cb_n$ in~\eqref{eq:tilde:Cbn}, but only on the interior of the unit hypercube since $\Cb_n/g^\omega$ and $\tilde \Cb_n/g^\omega$ are not bounded on the whole of $[0,1]^d$.

The aforementioned weighted weak convergence results allow us to prove the asymptotic validity of the subsampling methodology for the empirical copula processes considered in this work weighted by $1/g^\omega$. To be able to state the results for the processes $\Cb_n$ and $\tilde \Cb_n$, we need to introduce some additional notation first. Let
\begin{equation}
\label{eq:bar:Cbn}
\bar \Cb_n(\bm u) = \Gb_n(\bm u) - \sum_{j=1}^d \dot C_j(\bm u) \Gb_n(\bm u^{(j)}), \qquad \bm u \in [0,1]^d,
\end{equation}
where $\Gb_n$ is defined in~\eqref{eq:Gbn} and, for any $m \in \{1,\dots,N_{b,n}\}$, let
$$
\bar \Cb_b^{[m]}(\bm u) = \Gb_b^{[m]}(\bm u) - \sum_{j=1}^d \dot C_j(\bm u) \Gb_b^{[m]}(\bm u^{(j)}), \qquad \bm u \in [0,1]^d,
$$
where
\begin{equation}
\label{eq:Gbm}
\Gb_b^{[m]} = \sqrt{b} (G_b^{[m]} - G_n),
\end{equation}
and $G_b^{[m]}$ is the empirical d.f.\ of the subsample of size $b$ obtained from $\bm \Xc_b^{[m]}$ by (marginal) probability integral transformations; see~\eqref{eq:Ui}. The two following theorems are proven in Appendix~\ref{app:proofs}.

\begin{thm}[Subsampling the processes $\Cb_n / g^\omega$ and $\tilde \Cb_n / g^\omega$]
\label{thm:sub:Cbn:weighted}
Assume that Conditions~\ref{cond:partial:C} and~\ref{cond:partial2:C} hold, and that $b = b_n \to \infty$. Also, let $I_{1,n}$ and $I_{2,n}$ be independent random variables, independent of $\Xc_n$ and uniformly distributed on the set $\{1,\dots,N_{b,n}\}$.
\begin{enumerate}[(i)]
\item Under the i.i.d.\ setting, if $b/n \to \alpha \in [0,1)$, then, for any $\omega \in [0,1/2)$,
\begin{equation*}
(\bar \Cb_n / g^\omega, (1 - b/n)^{-1/2} \bar \Cb_b^{[I_{1,n}]} / g^\omega, (1 - b/n)^{-1/2}\bar \Cb_b^{[I_{2,n}]} / g^\omega) \leadsto (\Cb_C / g^\omega, \Cb_C^{[1]} / g^\omega,\Cb_C^{[2]} / g^\omega),
\end{equation*}
in $\{ \ell^\infty([0,1]^d) \}^3$, where $\Cb_C^{[1]}$ and $\Cb_C^{[2]}$ are independent copies of $\Cb_C$ in~\eqref{eq:CbC}.

\item Under the s.m.\ setting and Condition~\ref{cond:mixing}, if $b/n \to 0$, then, for any $\omega \in [0,1/2)$,
\begin{equation*}
(\bar \Cb_n / g^\omega, \bar \Cb_b^{[I_{1,n}]} / g^\omega, \bar \Cb_b^{[I_{2,n}]} / g^\omega) \leadsto (\Cb_C / g^\omega, \Cb_C^{[1]} / g^\omega, \Cb_C^{[2]} / g^\omega),
\end{equation*}
in $\{ \ell^\infty([0,1]^d) \}^3$. 
\end{enumerate}
Furthermore, under the assumptions in (i) or (ii), for any $c \in (0,1)$ and any $\omega \in [0,1/2)$,
\begin{align}
\label{eq:weighted:ae1}
\sup_{\bm u \in [0,1]^d \atop g(\bm u) \geq c/n} \left| \frac{\Cb_n(\bm u)}{\{ g(\bm u) \}^\omega} - \frac{\bar \Cb_n(\bm u)}{\{ g(\bm u) \}^\omega} \right| &= o_\Pr(1), \\
\label{eq:weighted:ae2}
\sup_{\bm u \in [0,1]^d \atop g(\bm u) \geq c/b} \left| \frac{\Cb_b^{[I_{m,n}]}(\bm u)}{\{ g(\bm u) \}^\omega} - \frac{\bar \Cb_b^{[I_{m,n}]}(\bm u)}{\{ g(\bm u) \}^\omega} \right| &= o_\Pr(1), \qquad m \in \{1,2\},
\end{align}
and
\begin{align}
\label{eq:weighted:tilde:ae1}
\sup_{\bm u \in [0,1]^d \atop g(\bm u) \geq c/n} \left| \frac{\tilde \Cb_n(\bm u)}{\{ g(\bm u) \}^\omega} - \frac{\bar \Cb_n(\bm u)}{\{ g(\bm u) \}^\omega} \right| &= o_\Pr(1), \\
\label{eq:weighted:tilde:ae2}
\sup_{\bm u \in [0,1]^d \atop g(\bm u) \geq c/b} \left| \frac{\tilde \Cb_b^{[I_{m,n}]}(\bm u)}{\{ g(\bm u) \}^\omega} - \frac{\bar \Cb_b^{[I_{m,n}]}(\bm u)}{\{ g(\bm u) \}^\omega} \right| &= o_\Pr(1), \qquad m \in \{1,2\},
\end{align}
where $\Cb_n$, $\tilde \Cb_n$, $\Cb_b^{[m]}$, $\tilde \Cb_b^{[m]}$, $m \in \{1,\dots,N_{b,n}\}$, are defined in~\eqref{eq:Cbn},~\eqref{eq:tilde:Cbn},~\eqref{eq:Cbbm} and~\eqref{eq:tilde:Cbbm}, respectively.
\end{thm}

\begin{remark}
The weak convergence $\bar \Cb_n / g^\omega \leadsto \Cb_C / g^\omega$ in $\ell^\infty([0,1]^d)$ for all $\omega \in [0,1/2)$ combined with~\eqref{eq:weighted:ae1} (or its $\tilde \Cb_n$ version~\eqref{eq:weighted:tilde:ae1}) is the slight extension of Theorem~2.2 in \cite{BerBucVol17} used in the proof of Theorem~2 in \cite{BerSeg18}. From the definition of $g$ in~\eqref{eq:g}, some thought reveals that, for any $c \in (0,1)$, the set $\{\bm u \in [0,1]^d : g(\bm u) \geq c/n \}$ contains the set $[c/n,1-c/n]^d$. The asymptotic equivalence in~\eqref{eq:weighted:ae1} (or in~\eqref{eq:weighted:tilde:ae1}) therefore holds on the set $[c/n,1-c/n]^d$ as stated in Theorem~2.2 in \cite{BerBucVol17}.
\end{remark}

\begin{thm}[Subsampling the processes $\Cb_n^\# / g^\omega$ and $\Cb_n^\beta / g^\omega$]
\label{thm:sub:Cbn:weighted:smooth}
Assume that Conditions~\ref{cond:partial:C} and~\ref{cond:partial2:C} hold and that $b = b_n \to \infty$. Also, let $I_{1,n}$ and $I_{2,n}$ be independent random variables, independent of $\Xc_n$ and uniformly distributed on the set $\{1,\dots,N_{b,n}\}$.
\begin{enumerate}[(i)]
\item Under the i.i.d.\ setting, if $b/n \to \alpha \in [0,1)$, then, for any $\omega \in [0,1/2)$,
\begin{align}
 \label{eq:sub:Cbn:hash:weighted:fsc}
(\Cb_n^\# / g^\omega, \Cb_{b,c}^{\#,[I_{1,n}]} / g^\omega, & \Cb_{b,c}^{\#,[I_{2,n}]} / g^\omega) \leadsto (\Cb_C / g^\omega, \Cb_C^{[1]} / g^\omega,\Cb_C^{[2]} / g^\omega), \\
\label{eq:sub:Cbn:beta:weighted:fsc}
(\Cb_n^\beta / g^\omega, \Cb_{b,c}^{\beta,[I_{1,n}]} / g^\omega, &\Cb_{b,c}^{\beta,[I_{2,n}]} / g^\omega) \leadsto (\Cb_C / g^\omega, \Cb_C^{[1]} / g^\omega,\Cb_C^{[2]} / g^\omega),
\end{align}
in $\{ \ell^\infty([0,1]^d) \}^3$, where $\Cb_C^{[1]}$ and $\Cb_C^{[2]}$ are independent copies of $\Cb_C$ in~\eqref{eq:CbC}, and $\Cb_{b,c}^{\#,[m]}$ and $\Cb_{b,c}^{\beta,[m]}$, $m \in \{1,\dots,N_{b,n}\}$, are the corrected versions of the subsample replicates defined in~\eqref{eq:Cbbm:hash} and~\eqref{eq:Cbbm:beta}, respectively.

\item Under the s.m.\ setting and Conditions~\ref{cond:ties} and~\ref{cond:mixing}, if $b/n \to 0$, then, for any $\omega \in [0,1/2)$,
\begin{align*}
(\Cb_n^\# / g^\omega, \Cb_b^{\#,[I_{1,n}]} / g^\omega, \Cb_b^{\#,[I_{2,n}]} / g^\omega) &\leadsto (\Cb_C / g^\omega, \Cb_C^{[1]} / g^\omega, \Cb_C^{[2]} / g^\omega), \\
(\Cb_n^\beta / g^\omega, \Cb_b^{\beta,[I_{1,n}]} / g^\omega, \Cb_b^{\beta,[I_{2,n}]} / g^\omega) &\leadsto (\Cb_C / g^\omega, \Cb_C^{[1]} / g^\omega, \Cb_C^{[2]} / g^\omega),
\end{align*}
in $\{ \ell^\infty([0,1]^d) \}^3$. 
\end{enumerate}
\end{thm}

\begin{remark}
A by-product of practical interest of the proof of Theorem~\ref{thm:sub:Cbn:weighted:smooth} is the weighted weak convergence of the empirical checkerboard copula process under the assumption of continuous marginal d.f.s $F_1,\dots,F_d$ considered in this work; see also Lemma~\ref{lem:wc:Cbn:hash:weighted} in Appendix~\ref{app:proofs}.
\end{remark}

We end this section by giving an example of application of Theorem~\ref{thm:sub:Cbn:weighted:smooth}. Following again~\citet[Section~4]{BerSeg18}, we consider the issue of estimating an extreme-value copula and state a result that confirms that basing the related inference on subsampling is asymptotically valid under the assumptions of Theorem~\ref{thm:sub:Cbn:weighted:smooth}

Let $\Delta_{d-1} = \{(w_1,\dots,w_{d-1}) \in [0,1]^{d-1} : w_1 + \dots + w_{d-1} \le 1 \}$ be the unit simplex. A copula $C$ is an extreme-value copula if and only if there exists a function $A:\Delta_{d-1} \to [1/d,1]$\index{$A$} such that, for any $\bm{u} \in (0,1]^d \setminus \{(1,\dots,1)\}$,
$$
C(\bm{u}) = \exp \biggl\{ \biggl(\,\sum_{j=1}^d \ln u_j \biggr) A \biggl(\frac{\ln u_2}{\sum_{j=1}^d \ln u_j}, \dots, \frac{\ln u_{d}}{\sum_{j=1}^d \ln u_j} \biggr)\biggr\}.
$$
The function $A$ is called the \emph{Pickands dependence function} associated with $C$. As explained, for instance, in~\cite[Section~4]{BerSeg18}, it can be expressed as a functional of $C$ through
\begin{equation}
  \label{eq:A:C}
  A(\bm w) = \nu(C)(\bm w), \qquad \bm w \in \Delta_{d-1},
\end{equation}
where the map $\nu$ from $\ell^\infty([0,1]^d)$ to $\ell^\infty(\Delta_{d-1})$ is defined, for any $f \in \ell^\infty([0,1]^d)$ and $\bm w \in \Delta_{d-1}$, by
\begin{equation}
  \label{eq:nu}
\nu(f)(\bm w) = \exp\left[ - \gamma + \int_0^1 \{ f(u^{w_1},\dots,u^{w_d}) - \1(u \in [e^{-1},1]) \} \frac{\dd u}{u \ln u} \right],
\end{equation}
with $\gamma = 0.5572\dots$ the Euler-Mascheroni constant.

Starting from~\eqref{eq:A:C}, a natural approach to obtain an estimator of $A$ consists of using the plug-in principle. Instead of replacing $C$ by the empirical copula $\hat C_n$ in~\eqref{eq:hat:Cn} (and thus obtaining the rank-based version of the \emph{Cap\'era\`a--Foug\`eres--Genest} estimator of $A$; see \cite{CapFouGen97,GenSeg09,GudSeg12}), \citet{BerSeg18} proposed to use the empirical beta copula $C_n^\beta$ in~\eqref{eq:Cn:beta}, which leads to the estimator
\begin{align*}
  A_n^{\beta}(\bm w) = \nu(C_n^\beta)(\bm w), \qquad \bm w \in \Delta_{d-1}.
\end{align*}
One could also consider the analogue estimator of $A$ based on the empirical checkerboard copula $C_n^\#$ in~\eqref{eq:Cn:hash}, namely
\begin{align*}
  A_n^{\#}(\bm w) = \nu(C_n^\#)(\bm w), \qquad \bm w \in \Delta_{d-1}.
\end{align*}

Let $\Ab_n^{\#} = \sqrt{n} (A_n^{\#} - A)$ and $\Ab_n^{\beta} = \sqrt{n} (A_n^{\beta} - A)$, and define their corresponding (corrected) subsample replicates by $\Ab_{b,c}^{\#,[m]} = (1-b/n)^{-1/2} \sqrt{b} \{\nu(C_b^{\#,[m]}) -  \nu(C_n^\#)\}$ and $\Ab_{b,c}^{\beta,[m]} = (1-b/n)^{-1/2} \sqrt{b} \{\nu(C_b^{\beta,[m]}) -  \nu(C_n^\beta)\}$, $m \in \{1,\dots,N_{b,n}\}$, respectively. The following result, proven in Appendix~\ref{app:proofs}, is then a consequence of Theorem~\ref{thm:sub:Cbn:weighted:smooth}.

\begin{cor}[Subsampling the processes $\Ab_n^\#$ and $\Ab_n^\beta$]
  \label{cor:sub:Abn}
  Let $C$ be an extreme-value copula with Pickands dependence function $A$.
  Under the assumptions of Theorem~\ref{thm:sub:Cbn:weighted:smooth},
  \begin{align}
\label{eq:sub:Abn:hash}
(\Ab_n^\#, \Ab_{b,c}^{\#,[I_{1,n}]}, & \Ab_{b,c}^{\#,[I_{2,n}]}) \leadsto (\Ab_C, \Ab_C^{[1]},\Ab_C^{[2]}), \\
\label{eq:sub:Abn:beta}
(\Ab_n^\beta, \Ab_{b,c}^{\beta,[I_{1,n}]}, &\Ab_{b,c}^{\beta,[I_{2,n}]}) \leadsto (\Ab_C, \Ab_C^{[1]},\Ab_C^{[2]}),
\end{align}
in $\{ \ell^\infty(\Delta_{d-1}) \}^3$, where $I_{1,n}$ and $I_{2,n}$ are independent random variables, independent of $\Xc_n$ and uniformly distributed on the set $\{1,\dots,N_{b,n}\}$, and $\Ab_C^{[1]}$ and $\Ab_C^{[2]}$ are independent copies of $\Ab_C$ defined by
$$
\Ab_C(\bm w) = A(\bm w) \int_0^1 \Cb_C(u^{w_1},\dots,u^{w_d}) \frac{\dd u}{u \ln u}, \qquad \bm w \in \Delta_{d-1}.
$$
\end{cor}

The previous corollary can be used, for example, to obtain an asymptotically valid symmetric $1-\alpha$ confidence band (see, e.g., \cite{Kos08}, Chapter~2) for $A$. Relying for instance on $A_n^\#$, such a confidence band is given by $A_n^\# \pm \Fc_{M}^{\#,-}(1-\alpha) / \sqrt{n}$, $\alpha \in (0,1/2)$, where $\Fc_{M}^\#$ is the empirical d.f.\ of a sample of $M$ (corrected) subsample replicates $\sup_{\bm w \in \Delta_{d-1}} |\Ab_{b,c}^{\#,[I_{m,n}]}(\bm w)|$, $m \in \{1,\dots,M\}$, of $\sup_{\bm w \in \Delta_{d-1}} |\Ab_{n}^{\#}(\bm w)|$, where $I_{1,n},\dots,I_{M,n}$ are chosen independently with replacement from $\{1,\dots,N_{b,n}\}$.

\section{Monte Carlo experiments}
\label{sec:MC}

The theoretical results provided in the preceding sections state conditions under which the subsampling methodology can be used to obtain asymptotically valid approximations of various (smooth, weighted) empirical copula processes. The results cover two types of subsampling: in the case of i.i.d.\ observations, subsamples of size $b < n$ are taken without replacement from the available data $\bm \Xc_n$; in the time series case, subsamples are restricted to consecutive observations to preserve serial dependence. In both cases, a crucial step prior to applying subsampling is the choice of the subsample size $b$.

\subsection{Subsampling for i.i.d.\ observations}

To investigate the influence of $b$ on the finite-sample performance of the subsampling methodology for the studied empirical copula processes in the case of i.i.d.\ observations, we considered an experimental setting similar to the one used in \cite{BucDet10}. Since all the empirical copula processes under consideration are rank-based, samples $\bm \Xc_n$ were generated directly from a $d$-dimensional copula $C$ chosen so that its bivariate margins have a Kendall's tau of $\tau$. For the $d$-dimensional copula $C$, we considered either a Clayton copula (which is lower-tail dependent) or a Gumbel--Hougaard copula (which is upper-tail dependent). The values of $n$, $d$ and $\tau$ were taken to vary in the sets $\{25,50,100,200,400\}$, $\{2,4\}$ and $\{0.33,0.66\}$, respectively. The experiments that were carried out are presented in detail hereafter along with a subset of representative results. More comprehensive results are available in the supplementary material.

\begin{table}[t!]
\centering
\caption{First horizontal block: (accurately estimated) covariance of $\hat \Cb_n$ at the points $P = \{(i/3,j/3):i,j=1,2\}$, for $n=100$, $C$ a bivariate Clayton copula and $\tau=0.33$. Remaining horizontal blocks: averages of 1000 covariance estimates based on subsampling (sub), empirical bootstrap (boot) and multiplier bootstrap (mult) approximations of $\hat \Cb_n$.}
\label{tab:cov:mean}

\begin{tabular}{llllll}
\toprule
  &   & (1/3, 1/3) & (1/3, 2/3) & (2/3, 1/3) & (2/3, 2/3)\\
\midrule
 & (1/3, 1/3) & 0.0488 & 0.0198 & 0.0200 & 0.0100\\

 & (1/3, 2/3) &  & 0.0337 & 0.0091 & 0.0185\\

 & (2/3, 1/3) &  &  & 0.0338 & 0.0185\\

\multirow{-4}{*}{\raggedright\arraybackslash $\hat \Cb_n$} & (2/3, 2/3) &  &  &  & 0.0513\\
\cmidrule{1-6}
 & (1/3, 1/3) & 0.0562 & 0.0205 & 0.0207 & 0.0089\\

 & (1/3, 2/3) &  & 0.0371 & 0.0084 & 0.0182\\

 & (2/3, 1/3) &  &  & 0.0375 & 0.0183\\

\multirow{-4}{*}{\raggedright\arraybackslash $\hat \Cb_{\ip{0.28n}}^{sub}$} & (2/3, 2/3) &  &  &  & 0.0583\\
\cmidrule{1-6}
 & (1/3, 1/3) & 0.0619 & 0.0241 & 0.0244 & 0.0096\\

 & (1/3, 2/3) &  & 0.0452 & 0.0094 & 0.0209\\

 & (2/3, 1/3) &  &  & 0.0458 & 0.0211\\

\multirow{-4}{*}{\raggedright\arraybackslash $\hat \Cb_n^{boot}$} & (2/3, 2/3) &  &  &  & 0.0690\\
\cmidrule{1-6}
 & (1/3, 1/3) & 0.0511 & 0.0199 & 0.0203 & 0.0092\\

 & (1/3, 2/3) &  & 0.0350 & 0.0091 & 0.0181\\

 & (2/3, 1/3) &  &  & 0.0356 & 0.0185\\

\multirow{-4}{*}{\raggedright\arraybackslash $\hat \Cb_n^{mult}$} & (2/3, 2/3) &  &  &  & 0.0536\\
\bottomrule
\end{tabular}

\end{table}

\paragraph{Subsampling approximation of the covariance of $\hat \Cb_n$ and choice of $b$}

Following \cite{BucDet10}, our first experiment, restricted to $d=2$, consisted of measuring how well the subsampling methodology can approximate the covariance of the empirical copula process $\hat \Cb_n$ in~\eqref{eq:hat:Cbn} at the points $P = \{(i/3,j/3):i,j=1,2\}$. We began by precisely estimating the covariance of $\hat \Cb_n$ at the points in $P$ from $100,000$ samples $\bm \Xc_n$. For $n=100$, $C$ a bivariate Clayton copula and $\tau=0.33$, these covariance values are given in the first horizontal block of Table~\ref{tab:cov:mean}. Next, for a given value of $b$ and each combination of $C$, $n$ and $\tau$, we generated 1000 samples $\bm \Xc_n$, and, for each sample, we computed $M=1000$ (corrected) subsample replicates $\hat \Cb_{b,c}^{[I_{1,n}]},\dots,\hat \Cb_{b,c}^{[I_{M,n}]}$ defined analogously to~\eqref{eq:Cbbm:corrected} from~\eqref{eq:hat:Cbbm}, and where $I_{1,n},\dots,I_{M,n}$ are independent random variables uniformly distributed on $\{1,\dots,N_{b,n}\}$. These $M=1000$ subsample replicates of $\hat \Cb_n$ were used to estimate the covariance of $\hat \Cb_n$ at the points in~$P$. For $n=100$, $C$ a bivariate Clayton copula, $\tau=0.33$ and $b = \ip{0.28n}$, the means over the samples $\bm \Xc_n$ of the 1000 covariance estimates are given in the second horizontal block of Table~\ref{tab:cov:mean}, while the corresponding (empirical) mean squared errors (MSEs) (with respect to the target values reported in the first horizontal block of Table~\ref{tab:cov:mean}) multiplied by~$10^4$ are given in the first horizontal block of Table~\ref{tab:cov:MSE}. To investigate the influence of $b$ on such MSEs, a grid of $b$ values was considered. We added to the grid the value $b = \ip{0.28n}$ suggested by \citet[Section~4]{Wu90} in the context of the delete-$h$ jackknife for the mean following an analysis based on Edgeworth expansions. The latter value was observed to give close to the lowest empirical MSEs across all our experiments in the i.i.d.\ setting (see the supplementary material for more details), which is why we chose to report results for this setting for $b$ in all subsequent experiments in the i.i.d.\ case. Note that choosing $b$ proportional to $n$ in such a way is completely compatible with the theoretical results stated in Theorems~\ref{thm:sub:Cbn}~(i),~\ref{thm:sub:Cbn:weighted}~(i) and~\ref{thm:sub:Cbn:weighted:smooth}~(i). Finally, it is important to mention that similar simulations (partly reported in the supplement) were carried using uncorrected subsample replicates and confirm that, overall, the finite population correction seems beneficial also in the context under consideration.

\begin{table}[t!]
\centering
\caption{For $n=100$, $C$ a bivariate Clayton copula and $\tau=0.33$, empirical MSEs ($\times 10^4$) of estimators of the covariance of $\hat \Cb_n$ at the points $P = \{(i/3,j/3):i,j=1,2\}$ based on subsampling (sub), the empirical bootstrap (boot) and the multiplier bootstrap (mult).}
\label{tab:cov:MSE}

\begin{tabular}{llllll}
\toprule
  &   & (1/3, 1/3) & (1/3, 2/3) & (2/3, 1/3) & (2/3, 2/3)\\
\midrule
 & (1/3, 1/3) & 0.9006 & 0.3521 & 0.3389 & 0.1907\\

 & (1/3, 2/3) &  & 0.8323 & 0.1147 & 0.1785\\

 & (2/3, 1/3) &  &  & 0.8250 & 0.1680\\

\multirow{-4}{*}{\raggedright\arraybackslash $\hat \Cb_{\ip{0.28n}}^{sub}$} & (2/3, 2/3) &  &  &  & 0.7330\\
\cmidrule{1-6}
 & (1/3, 1/3) & 2.2250 & 0.6925 & 0.6632 & 0.3199\\

 & (1/3, 2/3) &  & 2.2307 & 0.1995 & 0.3687\\

 & (2/3, 1/3) &  &  & 2.3528 & 0.3640\\

\multirow{-4}{*}{\raggedright\arraybackslash $\hat \Cb_n^{boot}$} & (2/3, 2/3) &  &  &  & 3.5708\\
\cmidrule{1-6}
 & (1/3, 1/3) & 0.6331 & 0.4966 & 0.4547 & 0.3144\\

 & (1/3, 2/3) &  & 0.9287 & 0.1811 & 0.2648\\

 & (2/3, 1/3) &  &  & 0.9084 & 0.2502\\

\multirow{-4}{*}{\raggedright\arraybackslash $\hat \Cb_n^{mult}$} & (2/3, 2/3) &  &  &  & 0.4086\\
\bottomrule
\end{tabular}

\end{table}

\paragraph{Comparison with the empirical bootstrap and multiplier bootstrap approximations}

In a second step, we carried out the same experiment using the classical empirical bootstrap and the multiplier bootstrap. The former consisted, for each generated sample $\bm \Xc_n$ from $C$, of generating $M=1000$ resamples (by sampling with replacement from $\bm \Xc_n$) and computing $\tilde \Cb_n$ in~\eqref{eq:tilde:Cbn} from each resample; the resulting $M=1000$ {\em resample replicates} of $\hat \Cb_n$ were used to estimate the covariance of $\hat \Cb_n$ at the points in $P$. For $n=100$, $C$ a bivariate Clayton copula and $\tau=0.33$, the averages of these estimates are given in the third horizontal block of Table~\ref{tab:cov:mean}, while the corresponding empirical MSEs are given in the second horizontal block of Table~\ref{tab:cov:MSE}. The multiplier bootstrap \citep[see, e.g.,][]{Sca05,RemSca09} was implemented as in~\cite{BucDet10} and the corresponding results are reported in the last horizontal blocks of Tables~\ref{tab:cov:mean} and~\ref{tab:cov:MSE} for $n=100$, $C$ a bivariate Clayton copula and $\tau=0.33$. Note that these two tables are directly comparable with the similar tables reported in \cite{BucDet10}. For all $C$ and $\tau$, the average of the 10 empirical MSEs is plotted against the sample size $n$ in Fig.~\ref{fig:cov:MSE}.

\begin{figure}[t!]
\begin{center}
  \includegraphics*[width=1\linewidth]{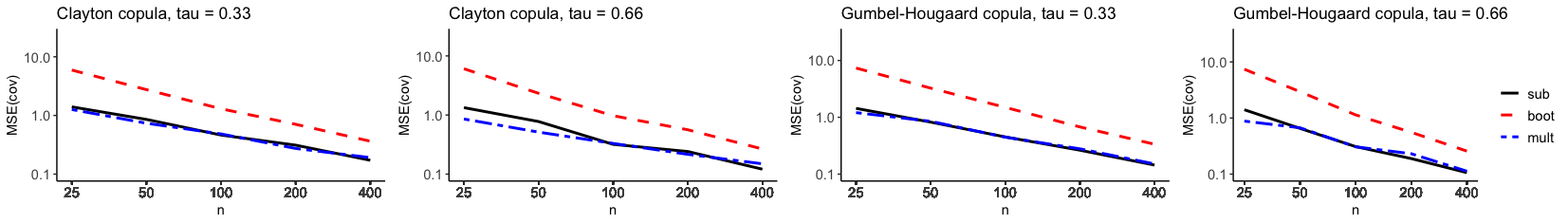}
  \caption{\label{fig:cov:MSE} Averages of the empirical MSEs ($\times 10^4$) of covariance estimators at the points $P = \{(i/3,j/3):i,j=1,2\}$ based on subsampling (sub), the empirical bootstrap (boot) and the multiplier bootstrap (mult) against the sample size $n$, for $C$ a bivariate Clayton or Gumbel--Hougaard copula and $\tau \in \{0.33,0.66\}$.}
\end{center}
\end{figure}

As one can see from Table~\ref{tab:cov:MSE} and Fig.~\ref{fig:cov:MSE}, the finite-sample performance of the subsampling approximation appears comparable to the one of the multiplier bootstrap and is substantially better than the one of the empirical bootstrap.

Similar simulations were also carried out for the $b$ out of $n$ bootstrap (see the supplementary material) which generalizes the empirical bootstrap. The smallest empirical MSEs, obtained when $b$ is ``small'' compared to $n$, were still higher, overall, than the corresponding empirical MSEs obtained using subsampling or the multiplier bootstrap. The latter should not come as a surprise, since, as already mentioned in the introduction, the $b$ out of $n$ bootstrap forms subsamples with ties, making it a biased resampling technique in the rank-based context under consideration.

\paragraph{Estimation of quantiles of Kolmogorov--Smirnov and Cramér--von Mises functionals}

Following again~\cite{BucDet10}, we focused next on subsampling, empirical bootstrap and multiplier bootstrap approximations of high quantiles of
\begin{equation}
  \label{eq:KS:CvM}
  KS(f_n) = \sup_{\bm u \in [0,1]^d} |f_n(\bm u)| \qquad \text{and} \qquad CvM(f_n) = \int_{[0,1]^d} \{ f_n(\bm u) \}^2 \dd \bm u
\end{equation}
for $f_n = \hat \Cb_n$. From a practical perspective, the supremum and the integral in~\eqref{eq:KS:CvM} were approximated by a maximum and a mean, respectively, using a uniform grid on $(0,1)^d$ of size $9^2$ when $d=2$ and $4^4$ when $d=4$. For every $d$, $C$, $\tau$ and $n$, the 90\% and 95\%-quantiles of $KS(\hat \Cb_n)$ and $CvM(\hat \Cb_n)$ were first estimated precisely from $100,000$ samples~$\bm \Xc_n$. For $n \in \{100,200\}$, $C$ a bivariate Clayton copula and $\tau=0.33$, these are given in the first lines of each horizontal block of Table~\ref{tab:stat:mean}. Then, for each $n$, 1000 samples~$\bm \Xc_n$ were generated and, for each $\bm \Xc_n$, one estimate of each quantile was computed from $M=1000$ subsampling, empirical bootstrap and multiplier bootstrap replicates of the considered functional. These estimations were also carried out using \emph{centered} replicates of $\hat \Cb_n$. In the case of subsampling, this consists of using, for any $\bm u \in [0,1]^d$ and $m \in \{1,\dots,M\}$,
\begin{equation}
\label{eq:hat:Cbb:fsc:centered}
\hat \Cb_{b,c}^{[I_{m,n}]}(\bm u) - \frac{1}{M} \sum_{m=1}^M \hat \Cb_{b,c}^{[I_{m,n}]}(\bm u),
\end{equation}
instead of $\hat \Cb_{b,c}^{[I_{m,n}]}(\bm u)$. The centered versions of the empirical bootstrap and multiplier bootstrap replicates are defined analogously. The rationale behind centering is that the replicates, whatever their type, converge weakly to copies of the \emph{centered} Gaussian process $\Cb_C$ in~\eqref{eq:CbC}. As the use of centered replicates always led to better finite-sample performance, it was adopted in all subsequent experiments. Notice that the use of centering is irrelevant in the previous covariance estimation experiment given the formula of the empirical covariance.

For each quantile and each type of approximation, the means of the 1000 estimates are reported in Table~\ref{tab:stat:mean}, while the corresponding MSEs are given in Table~\ref{tab:stat:MSE}. These two tables are again directly comparable with a similar table reported in \cite{BucDet10}. For $d=4$, $C$ a Gumbel--Hougaard copula and $\tau = 0.66$, for instance, the empirical MSEs of the quantile estimators are plotted against the sample size $n$ in Fig.~\ref{fig:stat:MSE}. Graphs for other $d$, $C$ and $\tau$ are not qualitatively different.

\begin{table}[t!]
\centering
\caption{First line of each horizontal block: (accurately estimated) 90\% and 95\%-quantiles of $KS(\hat \Cb_n)$ and $CvM(\hat \Cb_n)$ for $C$ a bivariate Clayton copula and $\tau=0.33$. Remaining lines of each horizontal block: averages of 1000 estimates of the same quantiles based on subsampling (sub), empirical bootstrap (boot) and multiplier bootstrap (mult) approximations of $\hat \Cb_n$.}
\label{tab:stat:mean}

\begin{tabular}{llrrrr}
\toprule
  & $f_{n}$ & $90\% (KS)$ & $95\% (KS)$ & $90\% (CvM)$ & $95\% (CvM)$\\
\midrule
 & $\hat \Cb_n$ & 0.5664 & 0.6437 & 0.0464 & 0.0580\\

 & $\hat \Cb_{\ip{0.28n}}^{sub}$ & 0.6209 & 0.6798 & 0.0465 & 0.0573\\

 & $\hat \Cb_n^{boot}$ & 0.6880 & 0.7526 & 0.0613 & 0.0743\\

\multirow{-4}{*}{\raggedright\arraybackslash $n = 100$} & $\hat \Cb_n^{mult}$ & 0.5964 & 0.6561 & 0.0478 & 0.0591\\
\cmidrule{1-6}
 & $\hat \Cb_n$ & 0.5770 & 0.6368 & 0.0463 & 0.0576\\

 & $\hat \Cb_{\ip{0.28n}}^{sub}$ & 0.6148 & 0.6744 & 0.0490 & 0.0605\\

 & $\hat \Cb_n^{boot}$ & 0.6549 & 0.7172 & 0.0555 & 0.0676\\

\multirow{-4}{*}{\raggedright\arraybackslash $n = 200$} & $\hat \Cb_n^{mult}$ & 0.5982 & 0.6576 & 0.0476 & 0.0590\\
\bottomrule
\end{tabular}

\end{table}

\begin{table}[t!]
\centering
\caption{For $C$ a bivariate Clayton copula and $\tau=0.33$, empirical MSEs ($\times 10^4$) of estimators of the 90\% and 95\%-quantiles of $KS(\hat \Cb_n)$ and $CvM(\hat \Cb_n)$ based on subsampling (sub), empirical bootstrap (boot) and multiplier bootstrap (mult) approximations of $\hat \Cb_n$.}
\label{tab:stat:MSE}

\begin{tabular}{llrrrr}
\toprule
  & $f_{n}$ & $90\% (KS)$ & $95\% (KS)$ & $90\% (CvM)$ & $95\% (CvM)$\\
\midrule
 & $\hat \Cb_{\ip{0.28n}}^{sub}$ & 34.7486 & 19.9671 & 0.2659 & 0.4342\\

 & $\hat \Cb_n^{boot}$ & 151.8726 & 124.1483 & 2.5361 & 3.1865\\

\multirow{-3}{*}{\raggedright\arraybackslash $n = 100$} & $\hat \Cb_n^{mult}$ & 13.8527 & 8.5551 & 0.3347 & 0.5792\\
\cmidrule{1-6}
 & $\hat \Cb_{\ip{0.28n}}^{sub}$ & 17.1279 & 17.9681 & 0.2344 & 0.3720\\

 & $\hat \Cb_n^{boot}$ & 63.5901 & 68.6258 & 1.0172 & 1.3002\\

\multirow{-3}{*}{\raggedright\arraybackslash $n = 200$} & $\hat \Cb_n^{mult}$ & 7.3044 & 8.4616 & 0.1850 & 0.3308\\
\bottomrule
\end{tabular}

\end{table}

As one can see from Table~\ref{tab:stat:MSE} and Fig.~\ref{fig:stat:MSE}, the quantile approximations based on subsampling are always better in terms of MSE than those based on the empirical bootstrap. They are similar (resp.\ slightly worse) than those based on the multliplier bootstrap for the Cramér--von Mises (resp.\ Kolmogorov--Smirnov) functional. 

An inspection of the more comprehensive simulation results presented in the supplementary material reveals that the subsampling approximations of the high quantiles of the Kolmogorov--Smirnov functional generally improve if $b$ is chosen smaller than $\ip{0.28n}$. For $b = \ip{0.1n}$ for instance, the approximations based on subsampling turn out to be comparable, overall, to those obtained using the multiplier bootstrap.

\begin{figure}[t!]
\begin{center}
  \includegraphics*[width=1\linewidth]{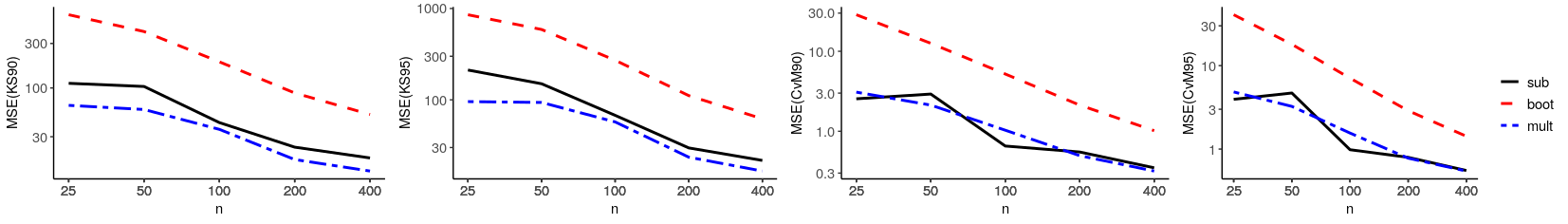}
  \caption{\label{fig:stat:MSE} Empirical MSEs ($\times 10^4$) of estimators of the 90\% and 95\%-quantiles of $KS(\hat \Cb_n)$ and $CvM(\hat \Cb_n)$ based on subsampling (sub), empirical bootstrap (boot) and multiplier bootstrap (mult) approximations of $\hat \Cb_n$ against the sample size $n$, for $d=4$, $C$ a Gumbel--Hougaard copula and $\tau = 0.66$.}
\end{center}
\end{figure}

\paragraph{Subsampling approximations of the smooth empirical copula processes}

To investigate the finite-sample performance of subsampling approximations of $\Cb_n^\#$ in~\eqref{eq:Cbn:hash} and $\Cb_n^\beta$ in~\eqref{eq:Cbn:beta}, we considered the same setting as in the first experiment. The goal was thus to estimate the covariances of $\Cb_n^\#$ and $\Cb_n^\beta$ at the points in $P$. Because these empirical copula processes should be closer and closer to $\hat \Cb_n$ in~\eqref{eq:hat:Cbn} as $n$ increases and given the high evaluation cost of the empirical beta copula $C_n^\beta$ defined in~\eqref{eq:Cn:beta}, the experiment was restricted to $n \in \{25,50,100\}$. For all bivariate $C$ and $\tau$, the average of the 10 MSEs is plotted against the sample size $n$ in Fig.~\ref{fig:cov:MSE:Smooth} for each of the three target processes $\hat \Cb_n$, $\Cb_n^\#$ and $\Cb_n^\beta$. As one can see, as $n$ increases, the three mean MSEs decrease and become closer and closer, as expected.

\begin{figure}[t!]
\begin{center}
  \includegraphics*[width=1\linewidth]{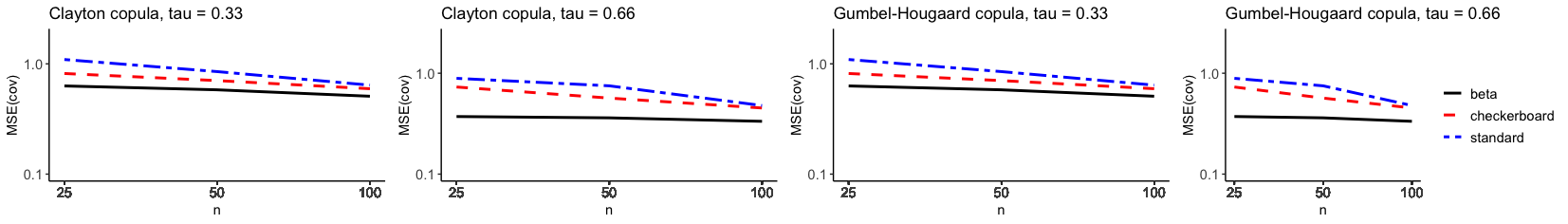}
  \caption{\label{fig:cov:MSE:Smooth} Averages of the empirical MSEs ($\times 10^4$) of subsampling-based estimators of the covariances of $\hat \Cb_n$, $\Cb_n^\#$ and $\Cb_n^\beta$ at the points $P = \{(i/3,j/3):i,j=1,2\}$ against the sample size $n$ for $C$ a Clayton or Gumbel--Hougaard copula and $\tau \in \{0.33,0.66\}$.}
\end{center}
\end{figure}

\subsection{Subsampling for time series}

To investigate the finite-sample performance of subsampling for approximating the studied empirical copula processes in a time series context, we considered a simple autoregressive model. Samples $\bm \Xc_n$ were generated as follows: Given a random sample $\bm U_i$, $i \in \{-100,\dots,0,\dots,n\}$, from a $d$-dimensional copula $C$, we formed the sample $\bm \epsilon_i = (\Phi^{-1}(U_{i1}),\dots,\Phi^{-1}(U_{id}))$, where $\Phi$ is the d.f.\ of the standard normal distribution, and set $\bm X_{-100} = \bm \epsilon_{-100}$. Next, given an autoregressive coefficient $\beta \in [0,1)$, we computed recursively
\begin{equation*}
X_{ij} = \beta X_{i-1,j} + \epsilon_{ij}, \qquad i \in \{-99,\dots,0,\dots,n\}, \, j \in \{1,\dots,d\},
\end{equation*}
and returned $\bm X_1,\dots, \bm X_n$.

Recall that, given such stretches $\bm \Xc_n$ from stationary time series and a subsample size $b < n$, subsamples $\bm \Xc_b^{[m]}$, $m \in \{1,\dots,N_{b,n}\}$, $N_{b,n} = n - b +1$, are restricted to consecutive observations to preserve serial dependence: they are of the form $\bm X_m,\dots,\bm X_{m+b-1}$.

Our experiments consisted of investigating the influence of the subsample size $b$ on the empirical MSEs of subsampling estimators of the 90\% and 95\%-quantiles of $KS(\hat \Cb_n)$ and $CvM(\hat \Cb_n)$. Fig.~\ref{fig:ts:b:MSE} displays such emiprical MSEs against $b$ for $C$ a bivariate Gumbel--Hougaard copula, $\tau = 0.33$, $\beta \in \{0,0.33,0.66\}$ and $n \in \{50,100,200\}$. An inspection of the $y$-axes of the graphs reveals that the MSEs increase with $\beta$ for $b$ and $n$ fixed, thereby empirically confirming that the stronger the serial dependence in the observations, the harder the estimation of the quantiles. In a related way, focusing on the curves for $n=200$, one can further notice that they are overall $u$-shaped and that their minima appear to shift to the right as $\beta$ increases, thereby empirically confirming the fact that, for fixed $n$, the ``optimal'' $b$ is expected to increase as the strength of the serial dependence increases.

As the setting $\beta = 0$ amounts to generating i.i.d.\ samples $\bm \Xc_n$, we finally aimed at empirically verifying that the aforementioned MSEs should be larger, overall, than if subsamples were formed by simply sampling without replacement from $\bm \Xc_n$ as in the case of i.i.d.\ observations. This is confirmed by Fig.~\ref{fig:ts:iid:MSE} which reports the empirical MSEs against $n$, for $C$ a bivariate Gumbel--Hougaard copula, $\tau = 0.33$ and $b \in \{3,7,11,15,19,23,27\}$. The solid lines give the empirical MSEs obtained when the subsampling is not restricted to consecutive observations. The results are not qualitatively different for other bivariate $C$ and $\tau$.

\begin{figure}[t!]
\begin{center}
   \includegraphics*[width=1\linewidth]{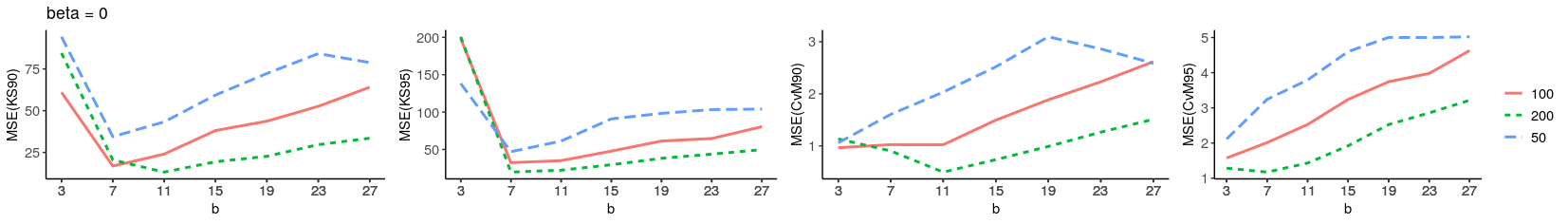}
  \includegraphics*[width=1\linewidth]{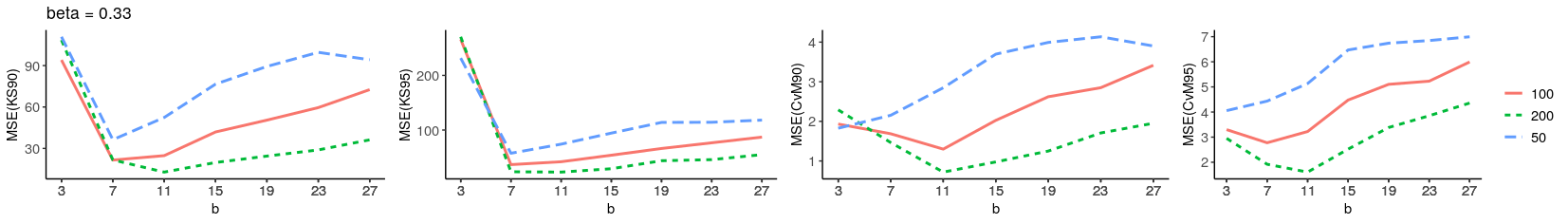}
  \includegraphics*[width=1\linewidth]{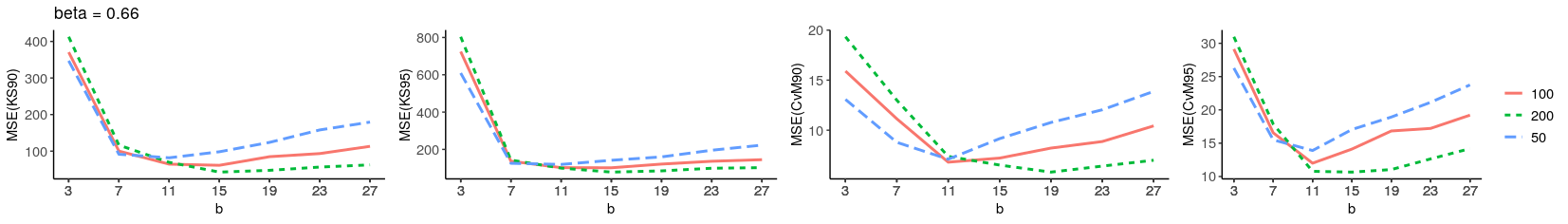}
  \caption{\label{fig:ts:b:MSE} Empirical MSEs ($\times 10^4$) of subsampling-based estimators of the 90\% and 95\%-quantiles of $KS(\hat \Cb_n)$ and $CvM(\hat \Cb_n)$ against $b$, for $C$ a bivariate Gumbel--Hougaard copula, $\tau = 0.33$, $n \in \{50,100,200\}$ and $\beta \in \{0,0.33,0.66\}$.}
\end{center}
\end{figure}

\begin{figure}[t!]
\begin{center}
  \includegraphics*[width=1\linewidth]{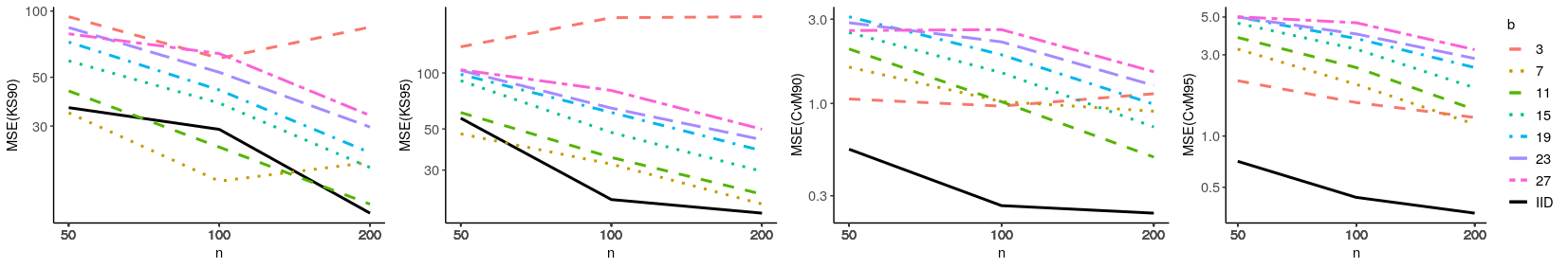}

  \caption{\label{fig:ts:iid:MSE} Empirical MSEs ($\times 10^4$) of subsampling-based estimators of the 90\% and 95\%-quantiles of $KS(\hat \Cb_n)$ and $CvM(\hat \Cb_n)$ against $n$, for $C$ a bivariate Gumbel--Hougaard copula, $\tau = 0.33$, $\beta = 0$ and $b \in \{3,7,11,15,19,23,27\}$. The solid lines give the empirical MSEs obtained in the case of subsampling for i.i.d.\ observations.}
\end{center}
\end{figure}

\section{Concluding remarks}
\label{sec:conc}

Relying on key results of \citet{PraWel93} in the i.i.d.\ case and \citet{PolRomWol99} in the time series case (under short range dependence), the asymptotic validity of subsampling was shown for various (weighted, smooth) empirical copula processes under minimalistic conditions. The results for the weighted empirical checkerboard and beta copula processes build up on the seminal work of \citet{BerBucVol17} and \citet{BerSeg18}, and seem to constitute first asymptotic validity results for bootstrapping these processes.

From a practical perspective, based on our numerous Monte Carlo experiments, we recommend to always use centered corrected subsample replicates as in~\eqref{eq:hat:Cbb:fsc:centered}, and, in the i.i.d.\ case, to consider the value $\ip{0.28n}$ as a starting choice for the subsample size $b$, as suggested by \citet[Section~4]{Wu90} in the case of the delete-$h$ jackknife for the mean. We were actually rather surprised not to find any mention of this proposal in the literature as it appears to be a rather natural initial choice when subsampling statistics or empirical processes converging weakly to Gaussian limits. Specifically, in our Monte Carlo experiments, the setting $b = \ip{0.28n}$ frequently lead to the best estimations of high quantiles of Cramér--von Mises functionals of the (standard) empirical copula process, while the setting $b = \ip{0.1n}$ was found to be better for Kolmogorov--Smirnov functionals. Overall, with $b \in \{\ip{0.1n}, \ip{0.28n}\}$, we observed the subsampling approximation of the (standard) empirical copula process to behave substantially better than its empirical bootstrap approximation, and to be roughly equivalent to its multiplier bootstrap approximation. As a consequence, subsampling appears as a natural, easier-to-implement alternative to the multiplier bootstrap in copula inference procedures in the i.i.d.\ case. Furthermore, as subsamples do not contain ties, it is of particular interest when dealing with statistics that can be expressed as functionals of (weighted) smooth empirical copula processes, given that the computation of such statistics is not fully meaningful in the presence of ties.

In the time series case, the choice of $b$ remains an open problem in the subsampling literature. Several practical solutions, of a more or less heuristic nature, are discussed in \cite[Chapter~9]{PolRomWol99b} and \cite{GluZihPol05}, and could be adapted to the copula inference setting under consideration. Once an efficient rule is found, it will be of practical interest to compare the resulting subsampling approximation of the empirical copula process $\hat \Cb_n$ in~\eqref{eq:hat:Cbn} with its dependent multiplier approximation as proposed in \cite{BucKoj16}.



\section{Appendix: Proofs}
\label{app:proofs}

\begin{lem}
\label{lem:ae:Cn:hash}
Assume that Condition~\ref{cond:ties} holds. Then, almost surely,
\begin{equation*}
\sup_{\bm u \in [0,1]^d} | C_n^\#(\bm u) - \hat C_n(\bm u) | \leq \frac{d}{n}.
\end{equation*}
\end{lem}

\begin{proof}[\bf Proof of Lemma~\ref{lem:ae:Cn:hash}]
  Let $\bm H_{n,r}(u) = \min\{ \max\{n u - r + 1, 0 \}, 1 \}$, $u \in [0,1]$, $r \in \{1,\dots,n\}$. With probability one, we have that
  \begin{align*}
    \sup_{\bm u \in [0,1]^d} | C_n^\#(\bm u) - \hat C_n(\bm u) | &\leq  \sup_{\bm u \in [0,1]^d}  \left| \frac{1}{n} \sum_{i=1}^n \prod_{j=1}^d  \1(R_{ij,n}/n \leq u_j) - \frac{1}{n} \sum_{i=1}^n \prod_{j=1}^d \bm H_{n,R_{ij,n}}(u_j) \right| \\
                                                            &\leq  \sup_{\bm u \in [0,1]^d} \frac{1}{n} \sum_{i=1}^n \left| \prod_{j=1}^d  \1(R_{ij,n}/n \leq u_j) -  \prod_{j=1}^d \bm H_{n,R_{ij,n}}(u_j) \right| \\
                                                            & \leq \sup_{\bm u \in [0,1]^d} \frac{1}{n} \sum_{i=1}^n  \sum_{j=1}^d \left| \1(R_{ij,n}/n \leq u_j) - \bm H_{n,R_{ij,n}}(u_j) \right| \\
                                                            & \leq \frac{1}{n} \sum_{j=1}^d \sup_{u \in [0,1]} \sum_{i=1}^n   \left| \1(R_{ij,n}/n \leq u) - \bm H_{n,R_{ij,n}}(u) \right| \\
                                                            & \leq \frac{1}{n} \sum_{j=1}^d \sup_{u \in [0,1]} \sum_{i=1}^n   \left| \1(i/n \leq u) - \bm H_{n,i}(u) \right| \\
                                                            & \leq \frac{d}{n} \sup_{u \in [0,1]}  \sum_{i=1}^n   \left| \1(i/n \leq u) - \bm H_{n,i}(u) \right| \\
                                                            & \leq \frac{d}{n} \max_{k \in \{1,\dots,n\}} \sup_{\frac{k-1}{n} \leq u \leq \frac{k}{n}}  \sum_{i=1}^n   \left| \1(i/n \leq u) - \bm H_{n,i}(u) \right| \\
                                                            & \leq \frac{d}{n} \max_{k \in \{1,\dots,n\}}  \sum_{i=1}^n \sup_{\frac{k-1}{n} \leq u \leq \frac{k}{n}}   \left| \1(i/n \leq u) - \bm H_{n,i}(u) \right| \\
                                                            & \leq \frac{d}{n} \max_{k \in \{1,\dots,n\}}  \sup_{\frac{k-1}{n} \leq u \leq \frac{k}{n}}   \left| \1(k/n \leq u) - \bm H_{n,k}(u) \right| \leq \frac{d}{n}.
  \end{align*}
\end{proof}

\begin{proof}[\bf Proof of Theorem~\ref{thm:sub:Cbn}]
Recall that $G_n$ is the empirical d.f.\ of the unobservable sample $\bm U_1,\dots,\bm U_n$ obtained from $\bm \Xc_n$ by~\eqref{eq:Ui}. For $m \in \{1,\dots,N_{b,n}\}$, let $\bm \Uc_b^{[m]}$ be the subsample of $\bm U_1,\dots,\bm U_n$ of size $b$ obtained from $\bm \Xc_b^{[m]}$ by the same marginal probability integral transformations. Furthermore, let $G_{n1},\dots,G_{nd}$ denote the univariate margins of $G_n$. It is well-known \citep[see, e.g.,][]{Seg12} that the empirical copula of $\bm \Xc_n$ can then be equivalently written as
$$
C_n(\bm u) = G_n\{G_{n1}^{-}(u_1),\dots,G_{nd}^{-}(u_d) \}, \qquad \bm u \in [0,1]^d,
$$
where $G_{nj}^{-}(v) = \inf\{u \in [0,1] : G_{nj}(u) \geq v \}$, $v \in [0,1]$, $j \in \{1,\dots,d\}$. For $m \in \{1,\dots,N_{b,n}\}$, let $G_b^{[m]}$ be the empirical d.f.\ of $\bm \Uc_b^{[m]}$ and let
$$
C_b^{[m]}(\bm u) = G_b^{[m]}\{G_{b1}^{[m],-}(u_1),\dots,G_{bd}^{[m],-}(u_d) \}, \qquad \bm u \in [0,1]^d,
$$
be the empirical copula of $\bm \Xc_b^{[m]}$. Furthermore, recall the definition of $\Gb_b^{[m]}$ in~\eqref{eq:Gbm} and that Condition~\ref{cond:Gbn} is assumed to hold.

\textbf{Proof of~\eqref{eq:sub:Cbn:fsc}:} 
From Theorem~2.2 and Example 3.6 in \cite{PraWel93} \citep[see also][Example~3.6.14]{vanWel96}, we have that, informally, ``$(1 - b/n)^{-1/2} \Gb_b^{[I_{1,n}]}$ converges weakly to $\Gb_C$ in $\ell^\infty([0,1]^d)$ conditionally on $\Xc_n$ in probability''; as mentioned earlier, see, e.g., Section~2.2.3 in \cite{Kos08} for a precise definition of that mode of convergence. From Lemma~3.1 in \cite{BucKoj18}, the latter is equivalent to
\begin{equation}
\label{eq:wc1}
(\Gb_n,(1 - b/n)^{-1/2} \Gb_b^{[I_{1,n}]},(1 - b/n)^{-1/2} \Gb_b^{[I_{2,n}]}) \leadsto (\Gb_C,\Gb_C^{[1]},\Gb_C^{[2]}),
\end{equation}
in $\{ \ell^\infty([0,1]^d) \}^3$, where $\Gb_C^{[1]}$ and $\Gb_C^{[2]}$ are independent copies of $\Gb_C$.

Let $\Db \subset \ell^\infty([0,1]^d)$ be the set of all d.f.s $H$ on $[0,1]^d$ whose univariate margins $H_j$, $j \in \{1,\dots,d\}$, satisfy $H_j(0) = 0$. Then, let $\Phi$ be the map from $\Db$ to $\ell^\infty([0,1]^d)$ defined by
\begin{equation}
\label{eq:Phi}
\Phi(H)(\bm u) = H \{ H_1^{-}(u_1),\dots,H_d^{-}(u_d) \}, \qquad \bm u \in [0,1]^d,
\end{equation}
where $H_{j}^{-}(v) = \inf\{u \in [0,1] : H_{j}(u) \geq v \}$, $v \in [0,1]$, $j \in \{1,\dots,d\}$. Since Condition~\ref{cond:partial:C} is assumed to hold, we have, from Theorem~2.4 of \cite{BucVol13}, that $\Phi$ is Hadamard-differentiable at $C$ tangentially to $\Cc_0$ given in~\eqref{eq:Cc0} with derivative
\begin{equation}
\label{eq:Phi'}
\Phi'_C(f)(\bm u) = f(\bm u) - \sum_{j=1}^d \dot C_j(\bm u) f(\bm u^{(j)}), \qquad \bm u \in [0,1]^d, \qquad f \in \ell^\infty([0,1]^d).
\end{equation}
Notice that $\Phi'_C$ is actually continuous on the whole of $\ell^\infty([0,1]^d)$ since $0 \leq \dot C_j(\bm u) \leq 1$, $\bm u \in [0,1]^d$, $j \in \{1,\dots,d\}$. Starting from~\eqref{eq:wc1} and using the continuous mapping theorem, we obtain that
$$
\Big( \Phi'_C(\Gb_n),\Phi'_C((1 - b/n)^{-1/2} \Gb_b^{[I_{1,n}]}),\Phi'_C((1 - b/n)^{-1/2} \Gb_b^{[I_{2,n}]}) \Big) \leadsto \Big(\Phi'_C(\Gb_C),\Phi'_C(\Gb_C^{[1]}),\Phi'_C(\Gb_C^{[2]})\Big) = (\Cb_C,\Cb_C^{[1]},\Cb_C^{[2]}),
$$
in $\{ \ell^\infty([0,1]^d) \}^3$. To prove~\eqref{eq:sub:Cbn:fsc}, it thus remains to show that
\begin{align}
\label{eq:first}
&\sup_{\bm u \in [0,1]^d} | \Cb_n(\bm u) - \Phi'_C(\Gb_n)(\bm u) | \p 0, \\
\label{eq:second}
&\sup_{\bm u \in [0,1]^d} | (1 - b/n)^{-1/2} \Cb_b^{[I_{m,n}]}(\bm u) - \Phi'_C\big((1 - b/n)^{-1/2} \Gb_b^{[I_{m,n}]}\big)(\bm u) | \p 0, \qquad m \in \{1,2\}.
\end{align}
Starting from Condition~\ref{cond:Gbn} and applying the delta method \citep[see][Theorem~3.9.4]{vanWel96} with the map $\Phi$ in~\eqref{eq:Phi}, we obtain that
\begin{equation}
\label{eq:delta}
\sup_{\bm u \in [0,1]^d} | \sqrt{n} \{ \Phi(G_n)(\bm u) - \Phi(C)(\bm u) \}  - \Phi'_C(\Gb_n)(\bm u) | \p 0,
\end{equation}
which is exactly~\eqref{eq:first}. 
Since $\Phi'_C$ is linear, by the triangle inequality,~\eqref{eq:second} is proven if
\begin{align}
\label{eq:third}
&\sup_{\bm u \in [0,1]^d} | \sqrt{b} \{ \Phi(G_b^{[I_{m,n}]})(\bm u) - \Phi(C)(\bm u) \} - \Phi'_C \big( \sqrt{b} ( G_b^{[I_{m,n}]} - C) \big) (\bm u)  | \p 0, \quad m \in \{1,2\},\\
\label{eq:fourth}
&\sup_{\bm u \in [0,1]^d} | \sqrt{b} \{ \Phi(G_n)(\bm u) - \Phi(C)(\bm u) \} - \Phi'_C \big( \sqrt{b} ( G_n - C) \big) (\bm u)  | \p 0.
\end{align}
The convergence in~\eqref{eq:fourth} then immediately follows from~\eqref{eq:delta}. To show~\eqref{eq:third}, we start from~\eqref{eq:wc1} and use the continuous mapping theorem to obtain that
\begin{equation}
\label{eq:wc:Gbm}
\sqrt{b} (G_b^{[I_{m,n}]} - C) = 
\Gb_b^{[I_{m,n}]} + \sqrt{b/n} \, \Gb_n \leadsto \sqrt{1 - \alpha} \, \Gb_C^{[m]} + \sqrt{\alpha} \, \Gb_C
\end{equation}
in $\ell^\infty([0,1]^d)$, for $m \in \{1,2\}$. Note in passing that $\sqrt{b} (G_b^{[I_{m,n}]} - C)$ and $\Gb_b = \sqrt{b} (G_b - C)$ are equal in distribution and so are their weak limits: the limiting process $\sqrt{1 - \alpha} \, \Gb_C^{[m]} + \sqrt{\alpha} \, \Gb_C$ is a tight, centered Gaussian process concentrated on $\Cc_0$ in~\eqref{eq:Cc0} whose covariance can be verified to be the same as the one of $\Gb_C$. The weak convergence in~\eqref{eq:wc:Gbm} can thus be combined with the delta method based on the map $\Phi$ in~\eqref{eq:Phi} to obtain~\eqref{eq:third}, which completes the proof of (i). 

\textbf{Proof of~\eqref{eq:sub:Cbn}:} We thus assume the s.m.\ setting and that $b / n \to 0$. In essence, the asymptotic validity of the subsampling methodology in this case is merely a consequence of Theorem~3.1 in \cite{PolRomWol99}. The proof below uses Theorem~4.1 in \cite{PolRomWol99} instead (a corollary of the aforementioned theorem), which will eventually allow us to conveniently apply Lemma~2.2 in~\cite{BucKoj18} to state the asymptotic validity under the form of a joint weak convergence with two subsample replicates.
  
Let $D([0,1]^d)$ be the space of \emph{càdlàg} functions on $[0,1]^d$ equipped with the Skorohod metric~$d_S$ that makes $(D([0,1]^d), d_S)$ separable and complete \citep[see, e.g.,][Chapter~3 for the case $d=1$]{Bil99}. Since $d_S$ is a weaker metric than the uniform metric, by the continuous mapping theorem, Condition~\ref{cond:Gbn} implies that $\Gb_n$ converges weakly to $\Gb_C$ in $D([0,1]^d)$ as well. Recall next that weak convergence in separable metric spaces can be metrized using the bounded Lipschitz metric; see, e.g., \citet[Theorem 11.3.3]{Dud02}, \citet[Section~2]{DumZer13} or \citet[Lemma 2.4]{BucKoj18}. The bounded Lipschitz metric $d_\BL$ between probability measures $P,Q$ on $D([0,1]^d)$ equipped with the Borel sigma field is defined by
\[
d_\BL(P,Q) = \sup_{f \in \BL_1(D([0,1]^d))} \big| \textstyle \int f \dd P - \int f \dd Q \big|,
\]
where $\BL_1(D([0,1]^d))$ denotes the set of functions $h:D([0,1]^d) \to [-1,1]$ such that $|h(x) - h(y)| \leq d_S(x,y)$ for all $x,y \in D([0,1]^d)$. Hence, denoting by $P^{\Gb_n}$ and $P^{\Gb_C}$ the probability measures of $\Gb_n$ and $\Gb_C$, respectively, the weak convergence of $\Gb_n$ to $\Gb_C$ in $D([0,1]^d)$ can be equivalently expressed as
\begin{equation}
\label{eq:BL1}
d_\BL(P^{\Gb_n},P^{\Gb_C}) \to 0.
\end{equation}

Let $P_{N_{b,n}}^{\Gb_n} = \frac{1}{N_{b,n}} \sum_{m=1}^{N_{b,n}} \delta_{\Gb_b^{[m]}}$ be the empirical measure of the $N_{b,n}$ subsample replicates $\Gb_b^{[1]},\dots,\Gb_b^{[N_{b,n}]}$ of $\Gb_n$.  Furthermore, let $M \in \N$ and let $I_{1,n},\dots,I_{M,n}$ be independent random variables, independent of $\Xc_n$ and uniformly distributed on the set $\{1,\dots,N_{b,n}\}$. Then, let $\hat P_{M}^{\Gb_n} = \frac{1}{M} \sum_{m=1}^{M} \delta_{\Gb_b^{[I_{m,n}]}}$ be the empirical measure of the $M$ subsample replicates $\Gb_b^{[I_{1,n}]},\dots,\Gb_b^{[I_{M,n}]}$, and note that $\hat P_{M}^{\Gb_n}$ is a \emph{random} probability measure on $D([0,1]^d)$ \cite[see, e.g.,][Section~2]{DumZer13}.

Next, let $f$ be a bounded and continuous function on $D([0,1]^d)$. For any $n \in \N$ and $\eps > 0$, since $\Gb_b^{[I_{1,n}]},\dots,\Gb_b^{[I_{M,n}]}$ are conditionally independent given~$\bm \Xc_n$, we have
\begin{align*}
\Pr\left( \left| \int f \, \dd \hat P_{M}^{\Gb_n} - \int f \, \dd P_{N_{b,n}}^{\Gb_n} \right| > \eps \right) &= \Pr\left\{ \left| \frac{1}{M} \sum_{m=1}^M f(\Gb_b^{[I_{m,n}]}) - \frac{1}{N_{b,n}} \sum_{m=1}^{N_{b,n}} f(\Gb_b^{[m]}) \right| > \eps \right\} \\
                                                                                                                 &= \Ex\left( \Pr\left[ \left| \frac{1}{M} \sum_{m=1}^M f(\Gb_b^{[I_{m,n}]}) - \Ex \{ f(\Gb_b^{[I_{1,n}]}) \} \right| > \eps \, \big| \, \bm \Xc_n \right] \right) \\
&\leq \frac{\Ex [ \Var \{ f(\Gb_b^{[I_{1,n}]}) \, | \, \bm \Xc_n\} ]}{\eps^2 M} \leq \frac{K}{\eps^2 M},
\end{align*}
by Chebychev's inequality and where $K$ is a bound on $f$. As a consequence,
\begin{equation}
\label{eq:BL2}
\int f \, \dd \hat P_{M}^{\Gb_n} - \int f \, \dd P_{N_{b,n}}^{\Gb_n} \p 0 \quad \text{as} \quad n,M \to \infty.
\end{equation}
Since Condition~\ref{cond:Gbn} is assumed to hold, from Theorem~4.1 of \cite{PolRomWol99}, we have that $d_\BL(P_{N_{b,n}}^{\Gb_n},P^{\Gb_n}) \to 0$, which, by the triangle inequality and~\eqref{eq:BL1} implies that $d_\BL(P_{N_{b,n}}^{\Gb_n},P^{\Gb_C}) \to 0$. The latter convergence further implies, for instance, by Lemma~2.4 in \cite{BucKoj18}, that $\int f \, \dd P_{N_{b,n}}^{\Gb_n} - \int f \, \dd P^{\Gb_C} \to 0$, which, combined with~\eqref{eq:BL2}, gives that $\int f \, \dd \hat P_{M}^{\Gb_n} - \int f \, \dd P^{\Gb_C} \p 0$ as $n,M \to 0$. Since $f$ was arbitrary, using, for instance, Lemma~2.5 in \cite{BucKoj18}, the previous convergence is equivalent to $d_\BL(\hat P_{M}^{\Gb_n},P^{\Gb_C}) \p 0$ as $n,M \to 0$, which can be combined with~\eqref{eq:BL1} to obtain that
\begin{equation}
\label{eq:BL3}
d_\BL(\hat P_{M}^{\Gb_n},P^{\Gb_n}) \p 0 \quad \text{as} \quad n,M \to \infty.
\end{equation}
By Lemma~2.2 in~\cite{BucKoj18}, the convergence in~\eqref{eq:BL3} is equivalent to the weak convergence of $(\Gb_n, \Gb_b^{[I_{1,n}]},\Gb_b^{[I_{2,n}]})$ to $(\Gb_C,\Gb_C^{[1]},\Gb_C^{[2]})$ in $\{D([0,1]^d)\}^3$, where $\Gb_C^{[1]}$ and $\Gb_C^{[2]}$ are independent copies of $\Gb_C$. Since the sample paths of the weak limit are (uniformly) continuous almost surely, the previous weak convergence occurs also in $\{ \ell^\infty([0,1]^d) \}^3$; see, e.g., \citet[Chapter~3]{Bil99} for the case $d=1$. The convergence in~\eqref{eq:sub:Cbn} then follows from the delta method based on the map $\Phi$ as in the proof of~\eqref{eq:sub:Cbn:fsc} but with $b/n \to \alpha = 0$.

\textbf{Proof of~\eqref{eq:sub:tilde:Cbn}:} The result is a straightforward consequence of~\eqref{eq:ae:Cn:tilde}. Indeed, from the latter asymptotic equivalence, we have that $\sup_{\bm u \in [0,1]^d} \sqrt{b} |C_b(\bm u) - \tilde C_b(\bm u) |$ converges in probability to zero. Since $(C_b, \tilde C_b)$ and $(C_b^{[I_{m,n}]}, \tilde C_b^{[I_{m,n}]})$, $m \in \{1,2\}$, are equal in distribution, we obtain that
  $$
  \sup_{\bm u \in [0,1]^d} \sqrt{b} |C_b^{[I_{m,n}]}(\bm u) - \tilde C_b^{[I_{m,n}]}(\bm u) | \p 0, \qquad m \in \{1,2\}.
  $$
  The desired result then follows from~\eqref{eq:sub:Cbn},~\eqref{eq:ae:Cn:tilde} and the previous display.

\textbf{Proofs of~\eqref{eq:sub:hat:Cbn:fsc} and~\eqref{eq:sub:hat:Cbn}:} The results immediately follow from~\eqref{eq:ae:Cn} using similar arguments.

\textbf{Proofs of~\eqref{eq:sub:Cbn:hash:fsc} and~\eqref{eq:sub:Cbn:hash}:} The results are direct consequences of Lemma~\ref{lem:ae:Cn:hash}, again, using similar arguments.

\textbf{Proofs of~\eqref{eq:sub:Cbn:beta:fsc} and~\eqref{eq:sub:Cbn:beta}:} We only prove~\eqref{eq:sub:Cbn:beta:fsc}, the proof of~\eqref{eq:sub:Cbn:beta} being simpler. We combine~\eqref{eq:third} and~\eqref{eq:wc:Gbm} to obtain that, for $m \in \{1,2\}$, $\sqrt{b} (C_b^{[I_{m,n}]} - C)$ converges weakly in $\ell^\infty([0,1]^d)$ to a limit process whose trajectories are continuous, almost surely. From~\eqref{eq:ae:Cn}, we immediately have that the same weak convergence occurs for the process $\sqrt{b} (\hat C_b^{[I_{m,n}]} - C)$. Using the fact that $\big(\sqrt{b} (C_b^{\beta,[I_{m,n}]} - C), \sqrt{b} (\hat C_b^{[I_{m,n}]} - C) \big)$ and $\big(\sqrt{b} (C_b^{\beta} - C), \sqrt{b} (\hat C_b - C) \big)$ are equal in distribution, we obtain from Corollary~3.7 in \cite{SegSibTsu17} that
\begin{equation}
\label{eq:ae1}
\sqrt{b} (C_b^{\beta,[I_{m,n}]} - C) = \sqrt{b} (\hat C_b^{[I_{m,n}]} - C) + o_\Pr(1), \qquad m \in \{1,2\}.
\end{equation}
From~\eqref{eq:sub:hat:Cbn:fsc}, we can apply the same corollary to obtain that
\begin{equation}
\label{eq:ae2}
\Cb_n^\beta = \hat \Cb_n + o_\Pr(1),
\end{equation}
which implies that
\begin{equation}
\label{eq:ae3}
\sqrt{b} (C_n^\beta - C) = \sqrt{b} (\hat C_n - C) + o_\Pr(1).
\end{equation}
Combining~\eqref{eq:ae1} and~\eqref{eq:ae3}, we obtain that
\begin{equation}
\label{eq:ae4}
\Cb_b^{\beta,[I_{m,n}]} = \hat \Cb_b^{[I_{m,n}]} + o_\Pr(1), \qquad m \in \{1,2\}.
\end{equation}
The weak convergence in~\eqref{eq:sub:Cbn:beta:fsc} is then an immediate consequence of~\eqref{eq:sub:hat:Cbn:fsc},~\eqref{eq:ae2} and~\eqref{eq:ae4}.
\end{proof}

\begin{lem}
\label{lem:sub:Cbn:weighted}
Assume that Conditions~\ref{cond:partial:C} and~\ref{cond:partial2:C} hold and that $b = b_n \to \infty$. Also, let $I_{1,n}$ and $I_{2,n}$ be independent random variables, independent of $\Xc_n$ and uniformly distributed on the set $\{1,\dots,N_{b,n}\}$. Then, under the i.i.d.\ setting with $b/n \to \alpha \in [0,1)$, or the s.m.\ setting and Condition~\ref{cond:mixing} with $b/n \to 0$, there holds, for any $\omega \in [0,1/2)$,
\begin{multline}
\label{eq:intermediate}
\Big( \Phi'_C(\Gb_n) / g^\omega, \Phi'_C(\sqrt{b} (G_b^{[I_{1,n}]} - C)) / g^\omega, \Phi'_C(\sqrt{b} (G_b^{[I_{2,n}]} - C)) / g^\omega \Big)
 \\ \leadsto
(\Cb_C / g^\omega, \sqrt{1 - \alpha} \, \Cb_C^{[1]} / g^\omega + \sqrt{\alpha} \, \Cb_C / g^\omega, \sqrt{1 - \alpha} \, \Cb_C^{[2]} / g^\omega + \sqrt{\alpha} \, \Gb_C / g^\omega)
\end{multline}
in $\{ \ell^\infty([0,1]^d) \}^3$, where $\Phi'_C$ is defined by~\eqref{eq:Phi'}, $\Gb_n$ is defined by~\eqref{eq:Gbn}, $g$ is defined by~\eqref{eq:g}, and $\Cb_C^{[1]}$ and $\Cb_C^{[2]}$ are independent copies of $\Cb_C$ in~\eqref{eq:CbC}.
\end{lem}

\begin{proof}[\bf Proof of Lemma~\ref{lem:sub:Cbn:weighted}]
We only provide the proof under the i.i.d.\ setting with $b/n \to \alpha \in [0,1)$, the proof being simpler when $b/n \to 0$. Since the assumptions are a superset of those of Theorem~\ref{thm:sub:Cbn}~(i), we can start from~\eqref{eq:wc1} and apply the continuous mapping theorem to obtain that
\begin{equation*}
(\Gb_n, \sqrt{b} (G_b^{[I_{1,n}]} - C), \sqrt{b} (G_b^{[I_{2,n}]} - C) ) \leadsto (\Gb_C,\sqrt{1 - \alpha} \, \Gb_C^{[1]} + \sqrt{\alpha} \, \Gb_C,\sqrt{1 - \alpha} \, \Gb_C^{[2]} + \sqrt{\alpha} \, \Gb_C)
\end{equation*}
in $\{ \ell^\infty([0,1]^d) \}^3$. Applying further the continuous mapping theorem with the linear map $\Phi'_C$ in~\eqref{eq:Phi'}, we obtain that
\begin{equation}
\label{eq:wc2}
\Big( \Phi'_C(\Gb_n), \Phi'_C(\sqrt{b} (G_b^{[I_{1,n}]} - C)), \Phi'_C(\sqrt{b} (G_b^{[I_{2,n}]} - C)) \Big) \leadsto (\Cb_C,\sqrt{1 - \alpha} \, \Cb_C^{[1]} + \sqrt{\alpha} \, \Cb_C,\sqrt{1 - \alpha} \, \Cb_C^{[2]} + \sqrt{\alpha} \, \Cb_C)
\end{equation}
in $\{ \ell^\infty([0,1]^d) \}^3$. The convergence of the finite-dimensional distributions in~\eqref{eq:intermediate} is then a consequence of~\eqref{eq:wc2} and the continuous mapping theorem. To show~\eqref{eq:intermediate}, it remains to prove marginal asymptotic tightness since the latter implies joint  asymptotic tightness. From Theorem~2.2 in \cite{BerBucVol17} (see also Lemma~4.9 in that reference and the discussion at the end of the proof of Theorem~2 in \cite{BerSeg18}), we have that, under the considered assumptions,
\begin{equation}
\label{eq:weighted:Cbn}
\Phi'_C(\Gb_n) / g^\omega \leadsto \Phi'_C(\Gb_C) / g^\omega = \Cb_C / g^\omega
\end{equation}
in $\ell^\infty([0,1]^d)$. Using the fact that $\sqrt{b} (G_b^{[I_{m,n}]} - C)$ and $\Gb_b = \sqrt{b} (G_b - C)$ are equal in distribution,~\eqref{eq:weighted:Cbn}, implies that
\begin{equation}
\label{eq:weighted:Cbn:2}
  \Phi'_C(\sqrt{b} (G_b^{[I_{m,n}]} - C)) / g^\omega \leadsto  \Phi'_C(\Gb_C) / g^\omega = \Cb_C / g^\omega,
\end{equation}
in $\ell^\infty([0,1]^d)$ for $m \in \{1,2\}$. Note in passing that, since, as already discussed in the proof of~\eqref{eq:sub:Cbn:fsc}, $\sqrt{1 - \alpha} \, \Gb_C^{[m]} + \sqrt{\alpha} \, \Gb_C$ and $\Gb_C$ are equal in distribution, $\Phi'_C(\sqrt{1-\alpha} \Gb_C^{[m]} + \sqrt{\alpha} \Gb_C ) = \sqrt{1-\alpha} \Cb_C^{[m]} + \sqrt{\alpha} \Cb_C$ and $\Phi'_C(\Gb_C) = \Cb_C$ are also equal in distribution. The weak convergences in~\eqref{eq:weighted:Cbn} and~\eqref{eq:weighted:Cbn:2} imply marginal asymptotic tightness of the process on the left-hand side of~\eqref{eq:intermediate} and thus the desired result.
\end{proof}

\begin{proof}[\bf Proof of Theorem~\ref{thm:sub:Cbn:weighted}]
The claims in (i) and (ii) are a consequence of Lemma~\ref{lem:sub:Cbn:weighted} and the continuous mapping theorem. The asymptotic equivalence in~\eqref{eq:weighted:tilde:ae1} follows from Theorem~2.2 in \cite{BerBucVol17}, as well as from the discussion at the end of the proof of Theorem~2 in \cite{BerSeg18}. From the same result, using the fact that $\big(\sqrt{b} ( G_b^{[I_{m,n}]} - C ), \sqrt{b} (\tilde C_b^{[I_{m,n}]} - C) \big)$ and $\big( \Gb_b = \sqrt{b} (G_b - C), \tilde \Cb_b = \sqrt{b} (\tilde C_b - C) \big)$ are equal in distribution for $m \in \{1,2\}$, we can also write
$$
\sup_{\bm u \in [0,1]^d \atop g(\bm u) \geq c/b} \left| \frac{\sqrt{b} \{ \tilde C_b^{[I_{m,n}]}(\bm u) - C(\bm u)\}}{\{ g(\bm u) \}^\omega} - \frac{\Phi'_C( \sqrt{b} ( G_b^{[I_{m,n}]} - C ) )(\bm u)}{\{ g(\bm u) \}^\omega} \right| = o_\Pr(1), \qquad m \in \{1,2\},
$$
where $\Phi'_C$ is defined by~\eqref{eq:Phi'}. Combining the previous statement with~\eqref{eq:weighted:tilde:ae1} and using the triangle inequality, we obtain~\eqref{eq:weighted:tilde:ae2}. Similarly, the asymptotic equivalences in~\eqref{eq:weighted:ae1} and~\eqref{eq:weighted:ae2} are essentially a consequence of Lemma~4.7 in~\cite{BerBucVol17}, the discussion at the end of the proof of Theorem~2 in \cite{BerSeg18} and Section~6.5 in the same reference.
\end{proof}

\begin{lem}
\label{lem:wc:Cbn:hash:weighted}
Assume that Conditions~\ref{cond:ties},~\ref{cond:partial:C},~\ref{cond:mixing} and~\ref{cond:partial2:C} hold. Then, for any $\omega \in [0,1/2)$,
$$
\Cb_n^\# / g^\omega = \bar \Cb_n / g^\omega + o_\Pr(1) \leadsto  \Cb_C / g^\omega
$$
in $\ell^\infty([0,1]^d)$, where $\Cb_C$, $g$ and $\bar \Cb_n$ are defined by~\eqref{eq:Gbn},~\eqref{eq:CbC} and~\eqref{eq:bar:Cbn}, respectively.
\end{lem}

\begin{proof}[\bf Proof of Lemma~\ref{lem:wc:Cbn:hash:weighted}]
The proof closely follows the proof of Theorem~2 in~\cite{BerSeg18}. Fix $\gamma \in \R$ such that $1/\{2(1-\omega)\} < \gamma < 1$ and consider the abbreviation $\{g \geq n^{-\gamma} \} = \{\bm u \in [0,1]^d : g(\bm u) \geq n^{-\gamma} \}$, and similarly for $\{g < n^{-\gamma} \}$. Then, write
$$
\Cb_n^\# / g^\omega = \1_{\{g \geq n^{-\gamma}\}} \Cb_n^\# / g^\omega  + \1_{\{g < n^{-\gamma}\}} \Cb_n^\# / g^\omega.
$$
Using the fact that $C_n^\#$ is a copula almost surely, we can proceed exactly as in the proof of Lemma~8 in~\cite{BerSeg18} to show that, almost surely,
$$
\sup_{\bm u \in [0,1]^d \atop g(\bm u) \leq n^{-\gamma}} | \Cb_n^\#(\bm u) / g^\omega(\bm u) | = o(1).
$$
Furthermore, from Lemma~\ref{lem:ae:Cn:hash},
\begin{align*}
\sup_{\bm u \in [0,1]^d \atop g(\bm u) \geq n^{-\gamma}} \left| \frac{\hat \Cb_n(\bm u)}{\{ g(\bm u) \}^\omega} - \frac{\Cb_n^\#(\bm u)}{\{ g(\bm u) \}^\omega} \right| \leq \sup_{\bm u \in [0,1]^d \atop g(\bm u) \geq n^{-\gamma}} \{ g(\bm u) \}^{-\omega} \times \sqrt{n} \sup_{\bm u \in [0,1]^d} | \hat C_n(\bm u) -C_n^\#(\bm u) | \leq  d  n^{\gamma \omega - 1/2} = o(1),
\end{align*}
almost surely. Combining the three previous displays, we obtain that
$$
\Cb_n^\# / g^\omega = \1_{\{g \geq n^{-\gamma}\}}  \hat \Cb_n / g^\omega  + o(1),
$$
almost surely, which, from~\eqref{eq:weighted:tilde:ae1} and the fact that $\tilde \Cb_n = \hat \Cb_n$ under Condition~\ref{cond:ties}, gives
$$
\Cb_n^\# / g^\omega = \1_{\{g \geq n^{-\gamma}\}}  \bar \Cb_n / g^\omega  + o_\Pr(1).
$$
From Lemma~4.10 in~\cite{BerBucVol17}, the indicator function on the right-hand side can be omitted and the desired result follows from Theorem~2.2 in the same reference.
\end{proof}

\begin{proof}[\bf Proof of Theorem~\ref{thm:sub:Cbn:weighted:smooth}]
  We only prove~\eqref{eq:sub:Cbn:hash:weighted:fsc} and~\eqref{eq:sub:Cbn:beta:weighted:fsc}, 
  the proofs of the other claims being simpler. Let us start with~\eqref{eq:sub:Cbn:beta:weighted:fsc}. From the last equation in the proof of Theorem~2 in \cite{BerSeg18}, we have that, under the considered assumptions,
\begin{equation}
\label{eq:weighted:beta:Cbn}
\Cb_n^\beta / g^\omega = \Phi'_C(\Gb_n) / g^\omega + o_\Pr(1). 
\end{equation}
Using the fact that $\big(\sqrt{b} (G_b^{[I_{m,n}]} - C), \sqrt{b} (C_b^{\beta, [I_{m,n}]} - C) \big)$ and $\big(\Gb_b = \sqrt{b} (G_b- C), \Cb_b^\beta =  \sqrt{b} (C_b^{\beta} - C) \big)$ are equal in distribution for $m \in \{1,2\}$, some thought reveals that~\eqref{eq:weighted:beta:Cbn} also implies that
\begin{equation}
\label{eq:weighted:beta:Cbn:2}
\sqrt{b} (C_b^{\beta, [I_{m,n}]} - C) / g^\omega =  \Phi'_C(\sqrt{b} (G_b^{[I_{m,n}]} - C)) / g^\omega + o_\Pr(1), \qquad m \in \{1,2\}.
\end{equation}
Combining~\eqref{eq:weighted:beta:Cbn} and~\eqref{eq:weighted:beta:Cbn:2} with~\eqref{eq:intermediate}, we obtain that
\begin{multline*}
(\Cb_n^\beta / g^\omega,\sqrt{b} (C_b^{\beta, [I_{1,n}]} - C) / g^\omega, \sqrt{b} (C_b^{\beta, [I_{2,n}]} - C) / g^\omega ) \\ \leadsto (\Cb_C / g^\omega, \sqrt{1 - \alpha} \, \Cb_C^{[1]} / g^\omega + \sqrt{\alpha} \, \Cb_C / g^\omega, \sqrt{1 - \alpha} \, \Cb_C^{[2]} / g^\omega + \sqrt{\alpha} \, \Cb_C / g^\omega)
\end{multline*}
in $\{ \ell^\infty([0,1]^d) \}^3$. The weak convergence in~\eqref{eq:sub:Cbn:beta:weighted:fsc} is finally mostly a consequence of the continuous mapping theorem. For the proof of~\eqref{eq:sub:Cbn:hash:weighted:fsc}, it suffices to start from Lemma~\ref{lem:wc:Cbn:hash:weighted} instead of~\eqref{eq:weighted:beta:Cbn}.
\end{proof}

\begin{proof}[\bf Proof of Corollary~\ref{cor:sub:Abn}]
  We only prove~\eqref{eq:sub:Abn:hash} as the proof of~\eqref{eq:sub:Abn:beta} is similar. Let $\omega \in (0,1/2)$ and let $\mu_\omega$ be the map from $\ell^\infty([0,1]^d)$ to $\ell^\infty(\Delta_{d-1})$ defined, for any $f \in \ell^\infty([0,1]^d)$ and $\bm w \in \Delta_{d-1}$, by
  $$
  \mu_\omega(f)(\bm w) = \int_0^1 f(u^{w_1},\dots,u^{w_d}) \{ g(u^{w_1},\dots,u^{w_d}) \}^\omega \frac{\dd u}{u \ln u},
  $$
  where $g$ is defined in~\eqref{eq:g}. As observed in \cite{BerSeg18}, since
$$
\sup_{\bm w \in \Delta_{d-1}} \left|  \int_0^1 \{ g(u^{w_1},\dots,u^{w_d}) \}^\omega \frac{\dd u}{u \ln u} \right| < \infty
$$
the map $\mu_\omega$ is continuous. Hence, from Theorem~\ref{thm:sub:Cbn:weighted:smooth} and the continuous mapping theorem,
\begin{multline}
  \label{eq:sub:wc:hash}
 \big( \mu_\omega(\Cb_n^\# / g^\omega), \mu_\omega(\Cb_{b,c}^{\#,[I_{1,n}]} / g^\omega), \mu_\omega(\Cb_{b,c}^{\#,[I_{2,n}]} / g^\omega) \big) \\ \leadsto
  \big( \mu_\omega(\Cb_C / g^\omega), \mu_\omega(\Cb_C^{[1]} / g^\omega), \mu_\omega(\Cb_C^{[2]} / g^\omega) \big) = \big( \eta(\Cb_C), \eta(\Cb_C^{[1]}), \eta(\Cb_C^{[2]}) \big),
\end{multline}
in $\{ \ell^\infty(\Delta_{d-1}) \}^3$, where, for any $f \in \ell^\infty([0,1]^d)$,
$$
\eta(f)(\bm w) = \int_0^1 f(u^{w_1},\dots,u^{w_d}) \frac{\dd u}{u \ln u}, \qquad \bm w \in \Delta_{d-1}.
$$
Furthermore, it can be verified that
$$
\mu_\omega(\Cb_n^\# / g^\omega) = \eta(\Cb_n^\#) = \sqrt{n} \{ \ln \nu(C_n^\#) - \ln \nu (C) \} = \sqrt{n} ( \ln A_n^\# - \ln A),
$$
where $\nu$ is defined in~\eqref{eq:nu}, and that, for any $m \in \{1,\dots,N_{b,n}\}$,
$$
\mu_\omega(\Cb_{b,c}^{\#,[m]} / g^\omega) = \eta(\Cb_{b,c}^{\#,[m]}) = (1-b/n)^{-1/2} \sqrt{b} \{ \ln \nu(C_b^{\#,[m]}) -  \ln \nu(C_n^\#)\}.
$$
The desired result finally follows from \eqref{eq:sub:wc:hash} and the delta method \citep[see][Theorem~3.9.4]{vanWel96} by proceeding, for instance, as in the proof of~Theorem~\ref{thm:sub:Cbn}~(i) for~\eqref{eq:sub:Cbn:fsc}.
\end{proof}

\section*{Acknowledgments}

The authors would like to thank two anonymous Referees for their constructive and insightful comments on an earlier version of this manuscript. This work was supported in part by the predoctoral grant PIFPAU17/04 of Kristina Stemikovskaya funded by the University of the Basque Country, the University of Pau and Pays de l'Adour and the Spanish Ministry of Economy, Industry and Competitiveness (TIN2016-78365-R).

\bibliographystyle{myjmva}
\bibliography{biblio}


\end{document}